\documentclass{amsart}
\usepackage{enumerate}
\usepackage{color}
\usepackage{amssymb}
\usepackage{amsmath}

\newtheorem{thm}{Theorem}[section]
\newtheorem{lem}[thm]{Lemma}
\newtheorem{prop}[thm]{Proposition}

\newtheorem{conv}[thm]{Convention}
\newtheorem{cor}[thm]{Corollary}
\newtheorem{rem}[thm]{Remark}
\newtheorem{fact}[thm]{Fact}
\newtheorem{construction}[thm]{Construction}
\newtheorem{example}[thm]{Example}

\newtheorem*{feco}{First Extrapolation Condition}
\newtheorem*{seco}{Second Extrapolation Condition}

\begin{document}

\title[Distributional inequalities for noncommutative martingales]{Distributional inequalities for noncommutative martingales}

\author[]{Yong Jiao}
\address{School of Mathematics and Statistics, Central South University, Changsha 410075, People's Republic of China}
\curraddr{}
\email{jiaoyong@csu.edu.cn}

\author[]{Fedor Sukochev}
\address{School of Mathematics and Statistics, University of New South Wales, Kensington,  2052, Australia}
\email{f.sukochev@unsw.edu.au}

\author[]{Lian Wu}
\address{School of Mathematics and Statistics, Central South University, Changsha 410075, People's Republic of China}
\email{wulian@csu.edu.cn}

\author[]{Dmitriy Zanin}
\address{School of Mathematics and Statistics, University of New South Wales, Kensington,  2052, Australia}
\email{d.zanin@unsw.edu.au}



\thanks{Corresponding author: Lian Wu}

\thanks{Yong Jiao is supported by the NSFC (No.11722114, No.11961131003); Fedor Sukochev and Dmitriy Zanin are supported by the Australian Research Council; Lian Wu is supported by the NSFC (No.11971484, No.11961131003)}

\begin{abstract} We establish distributional estimates for noncommutative martingales, in the sense of decreasing rearrangements of the spectra of unbounded operators, which generalises the study of distributions of random variables. Our results include distributional versions of the noncommutative Stein, dual Doob, martingale transform and Burkholder-Gundy inequalities. Our proof relies upon new and powerful extrapolation theorems. As an application, we obtain some new martingale inequalities in symmetric quasi-Banach operator spaces and some interesting endpoint estimates.
Our main approach demonstrates a method to build the noncommutative and classical probabilistic inequalities in an entirely operator theoretic way.
\end{abstract}

\maketitle

\section{Introduction}
The main purpose of the present paper is to study distributional estimates for noncommutative martingales. Recall that the investigation on noncommutative martingale inequalities originated from the work of Pisier and Xu \cite{PX2013}. Let $(\mathcal{M},\tau)$ be a von Neumann algebra, paired with a trace satisfying some continuity conditions. The pair $(\mathcal M,\tau)$ is called a noncommutative measure space, and extends the notion of a classical measure space to the study of spectra of operators affiliated to the von Neumann algebra $\mathcal M$. Let $(\mathcal{M}_k)_{k\geq0}$ be an increasing sequence of von Neumann subalgebras of $\mathcal{M}$ such that the union $\bigcup_{k\geq 0} \mathcal{M}_k$ is weak$^{\ast}$ dense in $\mathcal{M}.$ For $k\geq 0$, denote by $\mathcal{E}_{k}$ the conditional expectation with respect to $\mathcal{M}_k.$ The main result in \cite{PX2013} can be stated as follows: if $2\leq p<\infty$ and if $x\in L_p(\mathcal{M}),$ then
$$\|x\|_p \approx_p \Big\|\Big(\sum_{k\geq 0}x_k^*x_k\Big)^{1/2}\Big\|_p+\Big\|\Big(\sum_{k\geq 0}x_kx_k^*\Big)^{1/2}\Big\|_p,\eqno{\rm (BG)}$$
where $x_k=\mathcal{E}_{k}x-\mathcal{E}_{{k-1}}x $ (using the convention that $\mathcal{E}_{-1}=0$),
and where the notation $A\approx_p B$ means that there exists a constant $C_p$, dependent upon $p$, such that $C_p^{-1} B\le A\le C_p B$.
 The above equivalence is customarily referred to as noncommutative Burkholder-Gundy inequalities. Since the appearance of (BG), the theory of noncommutative martingale inequalities develops rapidly. A lot of classical martingale inequalities have been generalised to the noncommutative setting (see e.g. \cite{Junge-doob, JX2003, JX-best, Randrianantoanina, R2, R3}). Among these articles, the work due to Junge and Xu \cite{JX-best} is of special importance. It contains discussion on optimal orders for the best constants in the noncommutative Stein, dualised Doob and Burkholder-Gundy inequalities. In the light of the essential role these results perform in the proof of our main theorems, we review them below. As established in \cite[Theorem 8]{JX-best}, we have the following versions of Stein's inequality, the dualised Doob inequality, and estimates for martingale transforms:
\begin{enumerate}[\rm (i)]
\item {[Stein's inequality for the noncommutative $L_p$-spaces]} If $1< p<\infty$ and if $(x_k)_{k\geq 0}$ is a sequence in $L_p(\mathcal{M}),$ then
$$\Big\|\Big(\sum_{k\geq0}|\mathcal{E}_{k}x_k|^2\Big)^{\frac12}\Big\|_p\leq c_{{\rm abs}}\max\{p,p'\}\Big\|\Big(\sum_{k\geq0}|x_k|^2\Big)^{\frac12}\Big\|_p,\eqno{\rm (ST)}$$
where $p'=\frac{p}{p-1}$ is the conjugate index of $p.$ Here and in the following, the notation $c_{\rm abs}$ stands for an absolute constant.

\item {[Dualised Doob inequality for the noncommutative $L_p$-spaces]} If $1\leq p<\infty$ and if $(a_k)_{k\geq 0}$ is a sequence of positive operators in $L_p(\mathcal{M}),$ then
$$\big\|\sum_{k\geq0}\mathcal{E}_{k}a_k\big\|_p\leq c_{{\rm abs}}p^2\big\|\sum_{k\geq0}a_k\big\|_p.\eqno{\rm (DD)}$$

\item {[Martingale transform estimate for the noncommutative $L_p$-spaces]} Let $1<p<\infty$ and let $x\in L_p(\mathcal{M}).$ For every choice of signs $(\epsilon_k)_{k\geq0},$ we have
$$\Big\|\sum_{k\geq0}\epsilon_k\big(\mathcal{E}_{k}x-\mathcal{E}_{{k-1}}x\big)\Big\|_p\leq c_{{\rm abs}} \max\{p,p'\}\|x\|_p. \eqno{\rm (MT)}$$
\end{enumerate}

Strictly speaking, the paper \cite{JX-best} deals with noncommutative probability spaces. However, the estimates remain true in the setting of noncommutative measure spaces. Indeed, in the case of (DD), Junge and Xu refer to \cite{Junge-doob}, where the estimate is established for an arbitrary von Neumann algebra (not even semifinite). (ST) for $p\geq 2$ is an easy corollary of (DD) for $\frac{p}{2}.$ (ST) for $1<p\leq 2$ follows from (ST) for $p\geq2$ by duality. For (MT), Junge and Xu refer to \cite{Randrianantoanina}, where the estimate is established for semifinite von Neumann algebras.

All the constants that appeared in (ST), (DD) and (MT) are of optimal order within the realm of noncommutative probability spaces. Note that the optimal order for noncommutative martingale inequalities is sometimes different from that for the corresponding commutative inequalities.

Motivated by considerable progress made for martingale inequalities in noncommutative $L_p$-spaces,  more and more attention has been paid to exploring possible extensions of martingale inequalities for more general function spaces (such as noncommutative Lorentz spaces, noncommutative Orlicz spaces, symmetric Banach operator spaces, etc). We refer the reader to \cite{BCO, Dirksen2, Jiao1, Jiao2, JSZZ, RW, RW17, RWX} and references therein for more details.

A common feature of studying martingale inequalities in the spaces mentioned above is that one may often avoid having to deal with the distribution functions of measurable operators in question. This is in sharp contrast with the classical probability theory which typically operates with the estimates on the distribution function. To the best of our knowledge, such estimates have not appeared in noncommutative probability theory so far.

In this paper, we consider distributional estimates for noncommutative martingales. More precisely, we prove the following distributional versions of (ST), (DD), (MT) and (BG) (see Subsection \ref{svf subsection} for the definition of the singular value function (decreasing rearrangement) $\mu(x)$ of $\tau$-measurable operator and for the notion of Hardy-Littlewood-Polya submajorization $\prec\prec$; see also Subsection \ref{cesaro-calderon-def} for dual Ces\`{a}ro operator $C^{\ast}$ and the Calder\'{o}n operator $S$, and Subsection~\ref{lorentz subsection} for the definition of the Lorentz space $\Lambda_{\mathrm{log}}(\mathcal M)$.).
The following are our primary results, where the singular value function $\mu$ serves as a noncommutative analogue for the distribution function.

\begin{thm}\label{stein thm} {\rm [Distributional Stein's inequality]} For all sequences $(x_k)_{k\geq 0}\subset\Lambda_{{\rm log}}(\mathcal{M}),$ we have
$$\mu^{\frac12}\Big(\sum_{k\geq0}|\mathcal{E}_{k}x_k|^2\Big)\leq c_{{\rm abs}} \cdot S\mu^{\frac12}\Big(\sum_{k\geq0}|x_k|^2\Big).\eqno{\rm (DST)}$$
\end{thm}

\begin{thm}\label{dual doob thm} {\rm [Distributional dualised Doob inequality]} For all sequences of positive operators $(a_k)_{k\geq0}\subset(L_1+L_{\infty})(\mathcal{M})$ such that $(a_k^{\frac12})_{k\geq0}\subset \Lambda_{{\rm log}}(\mathcal{M}),$ we have
$$\mu\Big(\sum_{k\geq0}\mathcal{E}_{k}a_k\Big)\prec\prec c_{{\rm abs}} \Big(C^{\ast}\mu^{\frac12}\Big(\sum_{k\geq 0}a_k\Big)\Big)^2,\eqno{\rm (DDD)}$$
\end{thm}

\begin{thm}\label{mt theorem} {\rm [Distributional martingale transform estimate]} For all $x\in\Lambda_{{\rm log}}(\mathcal{M})$ and for every choice of signs $(\epsilon_k)_{k\geq0}$, we have
$$\mu\Big(\sum_{k\geq0}\epsilon_k\big(\mathcal{E}_{k}x-\mathcal{E}_{{k-1}}x\big)\Big)\leq c_{{\rm abs}}S\mu(x).\eqno{\rm (DMT)}$$
\end{thm}

\begin{thm}\label{bg theorem}  {\rm [Distributional Burkholder-Gundy inequalities]} Let $x\in\Lambda_{{\rm log}}(\mathcal{M})$ and let $(x_k)_{k\geq0}$ be the respective sequence of martingale differences (i.e., $x_0=\mathcal{E}_{_0} x$ and $x_k=\mathcal{E}_{k} x-\mathcal{E}_{k-1}x$ for $k\geq 1$).
\begin{enumerate}[{\rm (i)}]
\item if $x\in (L_2+L_{\infty})(\mathcal{M}),$ then
$$\mu\Big(\sum\limits_{k\geq 0}x_kx_k^{\ast}+x_k^{\ast}x_k\Big)\prec\prec c_{{\rm abs}} \big(C^{\ast}\mu(x)\big)^2\eqno{\rm (Lower -DBG)}.$$
\item if
$$\sum_{k\geq0}x_k^{\ast}x_k+x_kx_k^{\ast}\in (L_1+L_{\infty})(\mathcal{M}),\quad \big(\sum_{k\geq0}x_k^{\ast}x_k+x_kx_k^{\ast}\big)^{\frac12}\in\Lambda_{{\rm log}}(\mathcal{M}),$$
then $x\in (L_2+L_{\infty})(\mathcal{M})$ and
$$\mu^2(x)\prec\prec c_{{\rm abs}} \Big(C^{\ast}\mu^{\frac12}\Big(\sum\limits_{k\geq 0}x_kx_k^{\ast}+x_k^{\ast}x_k\Big)\Big)^2\eqno{\rm (Upper-DBG)}.$$
\end{enumerate}	
\end{thm}
The greatest advantage of these results is that they reduce the desired estimate to the verification of the corresponding boundedness of the involved operators $S$ and $C^*.$

Our method of proof mainly relies upon several optimal range theorems established in Section \ref{extrapolation section}. A related result, in more restrictive setting, was given in \cite[Theorem 14]{STZ}. In this paper, based on some results from Lykov \cite{Lykov} and the scaling technique, we provide new and more powerful optimal bound results, even when reduced to the classical case, which strengthen the Yano's extrapolation theorem \cite{Yano}.
According to these results, we can consider some standard inequalities with constants of optimal order through the lenses of operator theory and functional analysis, which allows one to construct the desired distributional inequalities. This demonstrates a method to build the noncommutative and classical probabilistic inequalities in an entirely operator theoretic way.

The paper is organised as follows. In Section \ref{prelims section}, we recall some background on the subject including any necessary notation and preliminary results. Section \ref{mar-sec} provides an equivalent description of a certain Marcinkiewicz norm. Noncommutative extrapolation theorems are stated and proved in Section \ref{extrapolation section}. These extrapolation results are important ingredients in the proof, key to showing the necessary optimal bounds, and are of independent interest. Section \ref{main section} is devoted to the proof of Theorems \ref{stein thm}, \ref{dual doob thm}, \ref{mt theorem} and \ref{bg theorem}. In Section \ref{optimal section}, we show that our results are optimal (at least for Stein inequality in the case $\mathcal{M}=B(L_2(0,\infty))$).  In Section \ref{application section}, we obtain new martingale inequalities in symmetric quasi-Banach operator spaces (including Orlicz and weak Orlicz spaces) as applications of our main results. The proof of Theorem \ref{bg theorem} relies on weak (1,1) estimates for generalised martingale transforms. We include the proof of these estimates in the Appendix.

\section{Preliminaries}\label{prelims section}

Throughout, we use $c_{abs}$ to denote some absolute constant which may change from line to line. We write $A\lesssim B$ if there is some absolute constant $c_{abs}$ such that $A\leq c_{abs} B$. We say that $A$ is equivalent to $B$ (written $A\thickapprox B$) if there exists some absolute constant $c_{\rm abs}$ such that $c_{\rm abs}^{-1} A\leq B\leq c_{\rm abs} A$. The notation $A\lesssim_{p}B$ (or, $A\thickapprox_p B$) means that the inequality (or, equivalence) holds true for some constant depending on the  parameter $p.$

\subsection{Generalised singular value functions}\label{svf subsection}

In what follows, $H$ is a separable Hilbert space and $\mathcal{M}\subset B(H)$ denotes a semifinite  von Neumann algebra equipped with a faithful normal semifinite trace $\tau.$ The pair $(\mathcal{M},\tau)$ is called a {\it noncommutative measure space}. A closed and densely defined operator $a$ on $H$ is said to be {\it affiliated} with $\mathcal{M}$ if $u^{\ast}au=a$ for each unitary operator $u$ in the commutant $\mathcal{M}'$ of $\mathcal{M}$. The operator $x$ is called {\it $\tau$-measurable} if $x$ is affiliated with $\mathcal{M}$ and for every $\varepsilon>0$, there exists a projection $p\in \mathcal{M}$ such that $p(H)\subset {\mathrm{dom}}(x)$ and $\tau(1-p)<\varepsilon.$
The set of all  $\tau$-measurable operators will be denoted by $L_0(\mathcal{M})$. Given a self adjoint operator $x\in L_0(\mathcal{M})$ and Borel set $B\subset\mathbb{R}$, we denote by $\chi_{B}(x)$ its spectral projection. For a projection $e\in \mathcal{M}$, if $\tau(e)<\infty$, we say that $e$ is $\tau$-finite.

Let $x=x^{\ast}\in L_0(\mathcal{M})$.  The {\it distribution function} of $x$ is defined by
$$n_x(s)=\tau\left(\chi_{(s,\infty)}(x)\right), \quad -\infty<s<\infty.$$
For $x=x^*,$ $y=y^*\in L_0(\mathcal M)$, we have (see e.g. \cite[Lemma 2.1]{JRWZ}):
\begin{equation}\label{distribution-triangle}
n_{x+y}(s+t)\leq n_{x}(s)+n_{y}(t),\quad -\infty<s,t<\infty.
\end{equation}
For $x\in L_0(\mathcal{M})$, the {\it generalised singular value function} of $x$ is defined by
$$\mu(t,x)=\inf\left\{s>0: n_{|x|}(s)\leq t\right\}, \quad t>0.$$
Similarly, for $x, y\in L_0(\mathcal{M})$, we also have (see e.g. \cite{FK}):
\begin{equation}\label{singular-triangle}
\mu(t+s,x+y)\leq \mu(t,x)+\mu(s,y),\quad s,t>0.
\end{equation}
The function $t\mapsto\mu(t,x)$ is decreasing and right-continuous. In the case that $\mathcal{M}$ is the abelian von Neumann algebra $L_{\infty}(0,\alpha)$ ($0<\alpha\leq \infty$) with the trace given by integration with respect to the Lebesgue measure,  $L_0(\mathcal{M})$ is the space of all measurable functions, with non-trivial distribution, and $\mu(f)$ is the decreasing rearrangement of a measurable function $f$; see  \cite{KPS}. In the abelian case, we write $L_0(0,\alpha)$ instead of
$L_0(L_\infty(0,\alpha))$ ($0<\alpha\leq \infty$). For more discussion on generalised singular value functions, we refer the reader to
\cite{FK, LSZ}.

Given $x, y\in L_0(\mathcal M)$, we say that $y$ is {\it submajorized} in the sense of Hardy-Littlewood-P\'{o}lya by $x$ (written $y\prec\prec x$) if
$$\int_0^t \mu(s,y)ds\leq \int_0^t \mu(s,x)ds,\quad t>0.$$

For $f\in L_0(0,\infty)$, the {\it dilation} operator $\sigma$ is defined by
$$\sigma_sf(t):=f\Big(\frac ts\Big),\quad s>0.$$

Let $(x_i)_{i\in I}\in L_0(\mathcal{M})$ be an increasing net of positive elements. If there exists $y\in L_0(\mathcal{M})$ such that $x_i\leq y$ for all $i\in I,$ then there exists $x=\sup_{i\in I}x_i$ (see Proposition 1.1 in \cite{DPS}). In this case, we write $x_i\uparrow x$ and say that the net $(x_i)_{i\in I}$ converges to $x$ in order. Proposition 1.7 in \cite{DPS} states that $\mu(x_i)\uparrow\mu(x)$ in such setting. Similarly, there exists $\inf_{i\in I}x_i$ of decreasing net $(x_i)_{i\in I}$ of positive elements in $L_0(\mathcal{M}).$ In this case, we write $x_i\downarrow x$ and also say that the net $(x_i)_{i\in I}$ converges to $x$ in order. If, in addition, each $x_i$ is bounded, then Vigier's theorem (see e.g. \cite[Theorem 4.1.1]{Murphy90}) states that $x_i\to x$ in strong operator topology.

\subsection{Symmetric function and operator spaces}
A Banach (or quasi-Banach) function space $(E,\|\cdot\|_E)$ on $(0,\alpha)$, $0<\alpha\leq \infty$ is called {\it symmetric} if for every $g\in E$ and for every measurable function $f$ with $\mu(f)\leq\mu(g)$, we have $f\in E$ and $\|f\|_E\leq\|g\|_E.$

We say that a (quasi-)norm $\|\cdot\|_E$ is order-continuous if
$$\inf_{i\in I}\|x_i\|_I=\|\inf_{i\in I}x_i\|_E$$
for an arbitrary decreasing net $(x_i)_{i\in I}\subset E$ of positive elements.

We say that $E$ possesses Fatou property if, for every increasing bounded net $(x_i)_{i\in I}\subset E$ of positive elements, we have
$$\sup_{i\in I}x_i\in E,\quad \|\sup_{i\in I}x_i\|_E=\sup_{i\in I}\|x_i\|_E.$$

The well known Lebesgue spaces $L_p$, Orlicz spaces $L_{\Phi}$ (see Subsection \ref{orlicz subsection}), Lorentz and Marcinkiewicz spaces (see Subsection \ref{lorentz subsection}) are symmetric, and have the Fatou property.

\begin{example} The symmetric space below has a special name ``weak $L_1$-space''. Define
$$L_{1,\infty}(0,\infty):=\left\{f\in L_0(0,\infty):\sup_{\lambda>0}\lambda n_{|f|}(\lambda)<\infty\right\}.$$
Then $L_{1,\infty}(0,\infty)$ becomes a symmetric quasi-Banach function space when equipped with the quasi-norm
$$\|f\|_{1,\infty}=\sup_{\lambda>0}\lambda n_{|f|}(\lambda)=\sup_{t>0}t\mu(t,f).$$
It can be verified that
$$\|f+g\|_{1,\infty}\leq 2\|f\|_{1,\infty}+2\|g\|_{1,\infty},\quad f,\,g\in L_{1,\infty}(0,\infty).$$
\end{example}

For $0<r<\infty,$  the \emph{$r$-convexification} of a (quasi-)Banach function space $E$ is defined by
\begin{equation*}
E^{(r)} :=\left\{ f \in L_0(0,\infty) : \big|f\big|^r \in E \right\}
\end{equation*}
equipped with the (quasi-)norm
 \begin{equation*}
 \|f \|_{E^{(r)}} = \big\| |f|^r\ \big\|_E^{\frac1r}.
\end{equation*}
It is easy to see that if $E$ is symmetric then so is  $E^{(r)}.$

Given a symmetric Banach function space $E$ on $(0,\infty)$,  the {\it K\"{o}the dual} of $E$ is defined by setting
$$E(0,\infty)^\times:=\left\{y\in L_0(0,\infty): \int_0^\infty|x(t)y(t)|dt<\infty,\,\,\forall x\in E(0,\infty)\right\}.$$
It can be verified that $E(0,\infty)^\times$ is a Banach space when equipped with the norm
$$\|y\|_{E^\times}:=\sup\left\{\int_0^\infty|x(t)y(t)|dt:x\in E(0,\infty),\|x\|_{E(0,\infty)}\leq 1\right\}.$$
The following general fact for K\"{o}the dual spaces is used later.
\begin{lem}\label{general koethe fact}
Let $E_1$, $E_2$ be symmetric Banach function spaces on $(0,\infty)$. We have
\begin{equation*}
(E_1+E_2)^\times=E_1^\times\cap E_2^\times,\quad (E_1\cap E_2)^\times = E_1^\times +E_2^\times.
\end{equation*}
\end{lem}

For a given symmetric Banach (or quasi-Banach) space $(E,\|\cdot\|_E)$, we define the corresponding non-commutative space on $(\mathcal{M},\tau)$ by setting (\cite{KS})
\begin{equation*}\label{ks def}
E(\mathcal{M},\tau):=\big\{x\in L_0(\mathcal{M}):\mu(x)\in E\big\}.
\end{equation*}
Endowed with the quasi-norm $\|x\|_{E(\mathcal M,\tau)}:=\|\mu(x)\|_E,$ the space $E(\mathcal{M},\tau)$ is called the noncommutative symmetric space associated with $(\mathcal{M},\tau)$ corresponding to the function space $(E,\|\cdot\|_E)$. It is shown in \cite{Su2014} that the quasi-normed space $(E(\mathcal{M},\tau),\|\cdot\|_{E(\mathcal M)})$ is complete if $(E,\|\cdot\|_E)$ is complete.

It is known from \cite{DPS} that $E(\mathcal{M},\tau)$ is order-continuous (respectively, has Fatou property) whenever $E$ is order-continuous (respectively, has the Fatou property).

When $E=L_p(0,\tau(1))$ for some $0<p<\infty$, then $E(\mathcal{M},\tau)$ coincides with the noncommutative Lebesgue space $L_p(\mathcal{M},\tau)$. In the case that $E=L_p+L_q$ for $0<p<q\leq \infty$, we have the following well known Holmstedt formula (see \cite[Theorem 4.1]{Holmstedt}):

\begin{lem}\label{lem-L1-Linfty} Let $0<p<q\leq \infty$. If $x\in (L _p+L_q)(\mathcal M,\tau)$, then
$$\|x\|_{(L_p+L_q)(\mathcal M,\tau)}\approx\Big(\int_0^1\mu(s,x)^pds\Big)^{1/p}+\Big(\int_1^\infty\mu(s,x)^qds\Big)^{1/q},\quad q<\infty$$
and
$$\|x\|_{(L_p+L_\infty)(\mathcal M,\tau)}\approx\Big(\int_0^1\mu(s,x)^pds\Big)^{1/p}.$$
\end{lem}

Similarly, we may define the K\"{o}the dual of  symmetric Banach operator spaces. Assume that $E$ is a symmetric Banach function space. The {\it K\"{o}the dual} of $E(\mathcal M,\tau)$ is defined (see \cite{DDP}) as
$$E(\mathcal M,\tau)^\times:=\left\{x\in L_0(\mathcal M,\tau): \tau(|xy|)<\infty,\,\, \forall y\in E(\mathcal M,\tau)\right\}$$
equipped with the norm
$$\|x\|_{E(\mathcal M,\tau)^\times}:=\sup\left\{|\tau(xy)|: y\in E(\mathcal M,\tau),\|y\|_{E(\mathcal M,\tau)}\leq 1\right\}.$$
It can be shown (see \cite{DDP}) that $(E(\mathcal M,\tau)^\times,\|\cdot\|_{E(\mathcal M,\tau)^\times})$ is a symmetric Banach operator space. Moreover, since $E$ is a symmetric Banach function space, we have
\begin{equation}\label{nc koethe eq}
E(\mathcal M,\tau)^\times=E^\times(\mathcal M,\tau).
\end{equation}
We also state the H\"{o}lder inequality here for further use.
\begin{lem}\label{holder} Let $E$ be a symmetric Banach function space. For $x\in E(\mathcal M,\tau)$, $y\in E(\mathcal M,\tau)^\times$, we have
$$\tau(|xy|)\leq \|x\|_{E(\mathcal M,\tau)}\|y\|_{E(\mathcal M,\tau)^\times}.$$
\end{lem}

In the sequel, without causing any confusion, we use $E(\mathcal M)$, $E^\times(\mathcal M)$, $\|\cdot\|_{E}$, $\|\cdot\|_{E^\times}$ to denote $E(\mathcal M,\tau)$, $E^\times(\mathcal M,\tau)$, $ \|\cdot\|_{E(\mathcal M)}$, $\|\cdot\|_{E^\times(\mathcal M)}$.

\subsection{Lorentz spaces and Marcinkiewicz spaces}\label{lorentz subsection}
Let $\phi:[0,\infty)\to [0,\infty)$ be an increasing concave continuous function such that $\lim_{t\to 0^+}\phi(t)=0$ and $\lim_{t\to \infty}\phi(t)=\infty$.
The {\it Lorentz} space $\Lambda_{\phi}$ is defined by setting
$$\Lambda_{\phi}(0,\infty):=\left\{x\in L_0(0,\infty):\int_0^\infty\mu(s,x)d\phi(s)<\infty\right\}$$
equipped with the norm
$$\|x\|_{\Lambda_{\phi}(0,\infty)}:=\int_0^\infty\mu(s,x)d\phi(s).$$
The {\it Lorentz sequence} space $\Lambda_{\phi}(\mathbb Z_+)$ is defined by setting
$$\Lambda_{\phi}(\mathbb Z_+):=\left\{a\in c_0(\mathbb Z_+):\|a\|_{\Lambda_{\phi}(\mathbb Z_+)} = \sum_{n=0}^\infty \mu(n,a)(\phi(n+1)-\phi(n))<\infty\right\},$$
where $c_0$ is the space of sequences converging to zero. The spaces $\Lambda_\phi(\mathbb R)$, $\Lambda_\phi(\mathbb Z)$ can be defined in a similar way. In the case $\phi(t)=\log(1+t)$, we use $\Lambda_{\log}(0,\infty)$, $\Lambda_{\log}(\mathbb Z_+)$, $\Lambda_{\log}(\mathbb R)$, $\Lambda_{\log}(\mathbb Z)$ to denote $\Lambda_{\phi}(0,\infty)$, $\Lambda_{\phi}(\mathbb Z_+)$, $\Lambda_\phi(\mathbb R)$, $\Lambda_{\phi}(\mathbb Z)$, respectively.
Note that Lorentz spaces have order-continuous norms.

The {\it Marcinkiewicz} space $M_{\phi}$ is defined by setting
$$M_{\phi}(0,\infty):=\left\{x\in L_0(0,\infty):\sup_{t>0}\frac1{\phi(t)}\int_0^t\mu(s,x)ds<\infty\right\}$$
equipped with the norm
$$\|x\|_{M_{\phi}(0,\infty)}:=\sup_{t>0}\frac1{\phi(t)}\int_0^t\mu(s,x)ds.$$
The {\it Marcinkiewicz sequence} space $m_{\phi}$  is defined as
$$m_{\phi}(\mathbb{Z}_+):=\Big\{z\in \ell_\infty(\mathbb{Z}_+):\|z\|_{m_\phi}=\sup_{k\geq 0}\frac1{\phi(k+1)}\sum_{j=0}^k\mu(j,z)<\infty\Big\}$$
where $\ell_\infty$ is the space of all bounded sequences.
In contrast to Lorentz spaces, Marcinkiewicz spaces need not have order-continuous norm.

We  collect some properties of Lorentz and Marcinkiewicz spaces below. For more detailed proof information, we refer the reader to \cite[Chapter II]{KPS}.
\begin{prop}\label{M-L-Properties}
 Let $\phi:[0,\infty)\to [0,\infty)$ be an increasing concave continuous function such that $\lim_{t\to 0^+}\phi(t)=0$ and $\lim_{t\to \infty}\phi(t)=\infty$.
\begin{enumerate}[\rm (i)]
\item The Lorentz space $\Lambda_{\phi}$ and the Marcinkiewicz space $M_{\phi}$ are symmetric.
\item The K\"othe dual of the Lorentz space $\Lambda_{\phi}$ is the Marcinkiewicz space $M_{\phi}.$
\item The K\"othe dual of the Marcinkiewicz space $M_{\phi}$ is the Lorentz space $\Lambda_{\phi}.$
\item Lorentz spaces and Marcinkiewicz spaces are closed with respect to the Hardy-Littlewood-P\'{o}lya submajorization.
\end{enumerate}
\end{prop}

\subsection{Orlicz and weak Orlicz spaces}\label{orlicz subsection}
By an {\it Orlicz function} $\Phi$ on $[0,\infty)$, we mean a continuous increasing convex function such that $\Phi(0)=0$ and $\lim_{t\to \infty} \Phi(t)=\infty.$ An Orlicz function $\Phi$ is said to be {\it $p$-convex} if the function $t \mapsto\Phi(t^{1/p})$ is convex, and to be {\it$q$-concave} if the function $t\mapsto \Phi(t^{1/q})$ is concave. For a given Orlicz function $\Phi,$ the associated {\it Orlicz space} $L_{\Phi}$ is defined by setting
$$L_{\Phi}(0,\infty):=\left\{f\in L_0(0,\infty):\ \int_0^\infty \Phi\Big(\frac{|f(s)|}{\lambda}\Big)ds<\infty\mbox{ for some }\lambda>0\right\}$$
equipped with the norm
$$\|f\|_\Phi:=\inf\left\{\lambda>0:\int_0^\infty \Phi\Big(\frac{|f(s)|}{\lambda}\Big)ds\leq 1\right\}.$$
The associated {\it weak Orlicz space} $L_{\Phi,\infty}$ is defined by
$$L_{\Phi,\infty}(0,\infty):=\left\{f\in L_0(0,\infty):\sup_\lambda \Phi(\lambda)n_{|f|}(c\lambda) <\infty\mbox{ for some }c>0\right\}$$
equipped the quasi-norm
$$\|x\|_{\Phi,\infty}:=\inf\left\{c>0:\sup_\lambda \Phi(\lambda)n_{|x|}(c\lambda)\leq 1\right\}.$$
Then $L_\Phi$ is a symmetric Banach function space and $L_{\Phi,\infty}$ is a symmetric quasi-Banach function spaces. We refer to \cite{Kras-Rutickii, Maligranda2} for more details on  Orlicz functions and  Orlicz spaces.

\subsection{Conditional expectations and martingales}
 Let us first recall a well known result on the existence of conditional expectation (see \cite{Umegaki} or \cite[Proposition 2.1]{DPPS2011}).

\begin{prop}\label{exp def prop} Let
\begin{enumerate}[{\rm (i)}]
\item $\mathcal{M}$ be a von Neumann algebra equipped with a normal, semifinite, faithful trace $\tau;$
\item $\mathcal{N}$ be a von Neumann subalgebra such that the restriction of $\tau$ to $\mathcal{N}$ is again semifinite;
\end{enumerate}
There exists a unique linear map $\mathcal{E}: (L_1+L_{\infty})(\mathcal{M})\to (L_1+L_{\infty})(\mathcal{N})$ such that
$$\tau(xy)=\tau(\mathcal{E}(x)y),\quad y\in (L_1\cap L_{\infty})(\mathcal{N}).$$
This map is called the conditional expectation. It satisfies the following properties:
\begin{enumerate}[{\rm (a)}]
\item $\mathcal{E}(x^{\ast})=(\mathcal{E}(x))^{\ast}$ for all $x\in (L_1+L_{\infty})(\mathcal{M});$
\item $\mathcal{E}(x)\geq0$ for all $0\leq x\in (L_1+L_{\infty})(\mathcal{M});$
\item if $0\leq x\in (L_1+L_{\infty})(\mathcal{M})$ is such that $\mathcal{E}(x)=0,$ then $x=0;$
\item $\mathcal{E}(x)=x$ for any $x\in (L_1+L_{\infty})(\mathcal{N});$
\item $\mathcal{E}(x^{\ast}x)\geq\mathcal{E}(x^{\ast})\cdot \mathcal{E}(x)$ for all $x\in (L_2+L_{\infty})(\mathcal{M});$
\item $\mathcal{E}(x)\in L_1(\mathcal{M})$ and $\tau(\mathcal{E}(x))=\tau(x)$ for all $x\in L_1(\mathcal{M});$
\item $\|\mathcal{E}(x)\|_1\leq \|x\|_1$ for all $x\in L_1(\mathcal{M})$ and $\|\mathcal{E}(x)\|_{\infty}\leq \|x\|_{\infty}$ for all $x\in L_{\infty}(\mathcal{M});$
\item $\mathcal{E}(xy)=x\mathcal{E}(y)$ $(\mathcal{E}(yx)=\mathcal{E}(y)x)$ for $x\in L_{\infty}(\mathcal{N}),$ $y\in L_1(\mathcal{M})$ (and also for $x\in L_1(\mathcal{N})$ and $y\in L_{\infty}(\mathcal{M})$).
\end{enumerate}
\end{prop}

The following assertion is claimed on p.211 in \cite{Takesaki2} and is, in fact, established in the proof of Theorem IX.4.2 there.

\begin{prop}\label{exp cpp prop} In the conditions of Proposition \ref{exp def prop}, the mapping $\mathcal{E}:\mathcal{M}\to\mathcal{N}$ is completely positive.
\end{prop}

Let $(\mathcal{M}_n)_{n\geq0}$ be an increasing sequence of von Neumann subalgebras of $\mathcal{M}$ such that the union $\bigcup_{n\geq 0}\mathcal M_n$ is weak$^{\ast}$ dense in $\mathcal{M}.$ Assume that for every $n\geq 0,$ the restriction $\tau|_{\mathcal{M}_n}$ is semifinite, so that there exists a trace preserving conditional expectation $\mathcal{E}_{n}$ from  $\mathcal{M}$ onto  $\mathcal{M}_n.$ Since each $\mathcal{E}_{n}$ preserves the trace, it extends to a contractive projection from $L_p(\mathcal{M},\tau)$ onto $L_p(\mathcal{M}_n,\tau)$ for all $1\leq p\leq \infty.$ Throughout the paper,  for convenience, we assume that $\mathcal{E}_{-1}=0.$

The first two statements are established in \cite[Theorem 2]{Tsukada}. The third one is a combination of the first two.

\begin{fact}\label{exp convergence fact} Let $\mathcal{E}_{n}$ be the conditional expectation as above.
\begin{enumerate}[{\rm (1)}]
\item If $x\in L_1(\mathcal{M})$, then $\mathcal{E}_{n}(x)\to x$ in $L_1(\mathcal{M}).$
\item If $x\in L_{\infty}(\mathcal{M})$, then  $\mathcal{E}_{n}(x)\to x$ ultraweakly.
\item If $x\in (L_1+L_{\infty})(\mathcal{M}),$ then $\mathcal{E}_{n}(x)\to x$ in $\sigma(L_1+L_{\infty},L_1\cap L_{\infty}).$
\end{enumerate}
\end{fact}

A sequence $(x_n)_{n\geq0}$ is called a noncommutative {\it martingale} with respect to $(\mathcal M_n)_{\geq 0}$ if $x_n\in (L_1+L_{\infty})(\mathcal{M}_n)$ for all $n\geq 0$ and
$$\mathcal{E}_{n-1}(x_n)=x_{n-1},\quad n\geq 1.$$
A sequence  $(x_n)_{n\geq0}$ is called a noncommutative {\it supermartingale} with respect to $(\mathcal{M}_n)_{n\geq0}$ if $x_n\in (L_1+L_\infty)(\mathcal{M}_n)$ and $$\mathcal{E}_{n-1}(x_{n})\leq x_{n-1},\quad n\geq1.$$
Let $1\leq p\leq \infty$ and let $(x_n)_{n\geq0}\subset L_p(\mathcal M)$ be a noncommutative martingale (or supermartingale). We say that $(x_n)_{n\geq0}$ is {\it $L_p$-bounded} if
$$\sup_{n\geq 0}\|x_n\|_p<\infty.$$
A sequence $(d_n)_{n\geq0}$ is called a sequence of noncommutative {\it martingale differences} if $d_n\in (L_1+L_{\infty})(\mathcal{M}_n)$ for all $n\geq0$ and
$$\mathcal{E}_{n-1}(d_n)=0,\quad n\geq 1.$$
For an element $y\in (L_1+L_{\infty})(\mathcal{M}),$ it is easy to check that the sequence $(\mathcal{E}_{n}y)_{n\geq0}$ is a  martingale.

\subsection{Ces\`{a}ro and Calder\'{o}n operators}\label{cesaro-calderon-def}
In this subsection, we recall some background on Ces\`{a}ro and Calder\'{o}n operators. The {\it Ces\`{a}ro} operator $C:(L_1+L_\infty)(0,\infty)\to (L_{1,\infty}+L_{\infty})(0,\infty)$ is defined by the formula
$$(Cf)(t):=\frac 1t\int_0^tf(s)ds,\quad f\in (L_1+L_\infty)(0,\infty).$$
It is easy to check that the operator $C^{\ast}$ defined by setting
$$(C^{\ast}f)(t):=\int_t^\infty\frac{f(s)}{s}ds,\quad f\in\Lambda_{\log}(0,\infty),$$
acts boundedly from $\Lambda_{\log}(0,\infty)$ to $(L_1+L_{\infty})(0,\infty).$
We may see this as follows.
\begin{fact} We have $\|C^\ast \|_{\Lambda_{\log}\to L_1+L_{\infty}}\leq c_{abs}$.
\end{fact}
\begin{proof}
Let $f\in \Lambda_{\log}(0,\infty)$. Using Lemma \ref{lem-L1-Linfty}, and noting that $C^\ast|f|$ is a decreasing function, we get
\begin{align*}
\|C^\ast f\|_{L_1+L_\infty}&\approx \int_0^1\mu(s,C^\ast f)ds\leq \int_0^1\mu(s,C^\ast |f|)ds =\int_0^1 (C^\ast |f|)(s)ds\\
&= \int_0^\infty\frac{|f(u)|}{u}\int_0^{\min\{1,u\}}dsdu=\int_0^\infty |f(u)|\min\{1,u\}\frac{du}{u}.
\end{align*}
Since the function $u\mapsto u^{-1}\min\{1,u\}$ is decreasing, it follows from \cite[Theorem II2.2, p.44]{BS1988} that
$$\int_0^\infty |f(u)|\min\{1,u\}\frac{du}{u}\leq \int_0^\infty \mu(u,f)\min\{1,u\}\frac{du}{u}.$$
Therefore,
\begin{align*}
\|C^\ast f\|_{L_1+L_\infty}& \leq\int_0^\infty \mu(u,f)\min\{1,u\}\frac{du}{u}\\
&=\int_0^1 \mu(u,f)du+\int_1^\infty \mu(u,f)\frac{du}{u}\\
&\approx \int_0^\infty \mu(u,f)(1+u)^{-1}du=\|f\|_{\Lambda_{\log}}.
\end{align*}
The assertion is proved.
\end{proof}

We refer to $C^{\ast}$ as to (formal) {\it dual of Ces\`{a}ro operator} due to the fact that
$$\langle Cf,g\rangle=\langle f,C^{\ast}g\rangle,\quad f,g\in L_2(0,\infty).$$

The {\it Calder\'{o}n} operator $S:\Lambda_{\log}(0,\infty)\to (L_1+L_{\infty})(0,\infty)$ is defined as the sum of $C$ and $C^\ast$, which is given by the formula
$$(Sf)(t):=\frac1t\int_0^t f(s)ds+\int_t^\infty\frac{f(s)}{s}ds,\quad f\in \Lambda_{\log}(0,\infty).$$
Similarly, the {\it discrete version of Ces\`{a}ro} operator $C_d:l_{\infty}(\mathbb{Z}_+)\to l_{\infty}(\mathbb{Z}_+)$ is given by the formula
$$(C_dx)(n)=\frac1{n+1}\sum_{k=0}^nx(k),\quad n\geq0.$$
The {\it discrete version of Calder\'{o}n} operator $S_d:\Lambda_{{\rm log}}(\mathbb{Z}_+)\to l_{\infty}(\mathbb{Z}_+)$ is defined as the sum of $C_d$ and $C_d^{\ast},$ which is given by the formula
$$(S_dx)(n)=\frac1{n+1}\sum_{k=0}^nx(k)+\sum_{k=n}^{\infty}\frac{x(k)}{k+1}.$$
We refer to \cite[Chapter III]{BS1988} and \cite[Chapter II]{KPS} for more discussion  on the above operators.

We end this subsection with the following  well known results for Ces\`{a}ro and Calder\'{o}n operators; see \cite[Theorem 327]{HLP} for the detailed proof.
\begin{lem}\label{cesaro-norm} The following are true:
$$\|C\|_{L_p\to L_p}=p',\quad 1<p\leq\infty,$$
$$\|C^{\ast}\|_{L_p\to L_p}=p,\quad 1\leq p<\infty,$$
$$\|S\|_{L_p\to L_p}\approx \max\{p,p'\}, \quad 1<p<\infty.$$

\end{lem}

\subsection{Triangular truncation operator}\label{truncation subsection}

Define the continuous triangular truncation as usual: if the operator $A$ is an integral operator on the Hilbert space $L_2(a,b),$ $-\infty\leq a<b\leq\infty,$ with the integral kernel $K,$ i.e.
$$(Af)(t)=\int_a^bK(t,s)f(s)ds,\quad t>0,\quad f\in L_2(a,b),$$
then $P(A)$ is an integral operator with truncated integral kernel
$$((P(A))f)(t)=\int_t^b K(t,s)f(s)ds,\quad t\in(a,b),\quad f\in L_2(a,b).$$

The properties of the operator $P$ were investigated in great detail in \cite{STZ}. In particular, the operator $P$ maps the ideal $\Lambda_{\log}(B(L_2(a,b)))$ into $B(L_2(a,b)).$ It is also stated there that $\Lambda_{\log}(B(L_2(a,b)))$ is the maximal domain of $P$ (the proof in \cite{STZ} contains a mistake, fixed in \cite{STZ-corr}).

\section{Marcinkiewicz spaces and conditions on momenta}\label{mar-sec}

In this section, we demonstrate that a random variable belongs to a Marcinkiewicz space if and only if certain condition on its momenta holds. This assertion plays a key role in the proof of extrapolation theorems in the next section.

In preparation for the proof of Theorem  \ref{first extrapolation theorem} and Corollary \ref{first extrapolation corollary}, we introduce the following concave increasing function:
\begin{equation}\label{def-phi}\phi(t):=
\begin{cases}
t\log(\frac{e^2}{t}),& t\in[0,1],\\
\log(e^2t),& t\in(1,\infty).
\end{cases}
\end{equation}

\begin{prop}\label{momenta proposition} For every function $x\in L_0(0,\infty),$ we have
$$\sup_{1<p<\infty}\min\{p^{-1},p'^{-1}\}\|x\|_p\approx\|x\|_{M_{\phi}}.$$
\end{prop}

The proof of Proposition~\ref{momenta proposition} is delayed until we first show some necessary function space isomorphism lemmas.

In the following lemma, we describe elements of the Orlicz space $L_{\Phi}(0,1)$ in terms of their moments.

\begin{lem}\label{momenta 01} For every function $y\in L_0(0,1),$ we have
$$\sup_{2\leq p<\infty}p^{-1}\|y\|_p\approx\|y\|_{L_{\Phi}(0,1)}.$$
Here, $\Phi(t)=e^t-1,$ $t>0.$
\end{lem}
\begin{proof} Fix a measurable function $y$ on the interval $(0,1).$
	
Assume first that $\|y\|_{L_{\Phi}}\leq 1$ or, equivalently, $\|\Phi(|y|)\|_1\leq 1.$ Observe that
$$\frac{|y|^k}{k^k}\leq\frac{|y|^k}{k!}\leq \Phi(|y|),\quad k\in\mathbb{N}.$$
Therefore, $$\frac{\|y\|_k^k}{k^k}=\Big\|\frac{|y|^k}{k^k}\Big\|_1\leq\|\Phi(|y|)\|_1\leq 1, \quad k\in\mathbb{N}.$$
Taking the $k$-th root, we obtain
$$\frac{\|y\|_k}{k}\leq 1,\quad k\in\mathbb{N}.$$
For every $2\leq p<\infty,$ there exists $k\in\mathbb{N}$ such that $k\leq p<2k.$ It follows that
$$\frac{\|y\|_p}{p}\leq\frac{\|y\|_{2k}}{p}\leq\frac{\|y\|_{2k}}{k}\leq 2.$$
This proves that
\begin{equation}\label{expl1 eq1}
\sup_{2\leq p<\infty}\frac1p\|y\|_p\leq 2\|y\|_{L_{\Phi}}.
\end{equation}

Conversely, suppose that
$$\sup_{2\leq p<\infty}\frac1p\|y\|_p\leq\frac1{2e}.$$
It is immediate that
$$\||y|^k\|_1=\|y\|_k^k\leq \frac{k^k}{(2e)^k},\quad 1\neq k\in\mathbb{N}.$$
Also,
$$\||y|\|_1\leq\|y\|_2\leq\frac1e.$$
Therefore,
$$\|\Phi(|y|)\|_1=\sum_{k=1}^{\infty}\frac{\||y|\|_k^k}{k!}\leq\frac1e+\sum_{k=2}^{\infty}\frac{k^k}{(2e)^k\cdot k!}.$$
By Stirling inequality, we have
$$k!\geq \sqrt{2\pi k}(\frac{k}{e})^k\geq(\frac{k}{e})^k,\quad k\geq1.$$
we obtain
$$\|\Phi(|y|)\|_1\leq\frac1e+\sum_{k=2}^{\infty}2^{-k}=\frac1e+\frac12<1.$$
This means $\|y\|_{L_{\Phi}}\leq 1.$ Hence,
\begin{equation}\label{expl1 eq2}
\|y\|_{L_{\Phi}}\leq 2e\sup_{2\leq p<\infty}p^{-1}\|y\|_p.
\end{equation}

A combination of the estimates \eqref{expl1 eq1} and \eqref{expl1 eq2} yields the claim.
\end{proof}

\begin{lem}\label{momenta seq} For every $a\in \ell_{\infty}$, we have
$$\sup_{1<p\leq 2}(p-1)\|a\|_p\approx\|a\|_{m_{{\rm log}}}.$$
\end{lem}
\begin{proof} This assertion is a special case of \cite[Corollary 3]{Lykov} (see also a very similar earlier result in \cite[Theorem 4.5]{CRSS}). Indeed, take $\alpha=1$ in \cite[Corollary 3]{Lykov}. We have
$$\sup_{1<p\leq 2}(p-1)\|a\|_p \approx \sup_{n\in\mathbb{Z}_+}\frac{1}{\log(e(n+1))} \sum_{k=0}^n\mu(k,a)\approx\|a\|_{m_{{\rm log}}}.$$
\end{proof}

\begin{cor}\label{momenta cor} For every function $x\in L_0(0,\infty),$ we have
$$\sup_{1<p\leq 2}(p-1)\|x\|_p\approx\|\mu(x)\chi_{(0,1)}\|_2+\|(\mu(k,x))_{k\geq1}\|_{m_{{\rm log}}},$$
$$\sup_{2\leq p<\infty}\frac1p\|x\|_p\approx\|\mu(x)\chi_{(0,1)}\|_{L_{\Phi}(0,1)}+\|(\mu(k,x))_{k\geq1}\|_{\ell_2}.$$
\end{cor}
\begin{proof} For any $1<p<\infty$, it is clear that
$$\|x\|_p\geq\|(\mu(k,x))_{k\geq1}\|_p,\quad \|x\|_p\geq \|\mu(x)\chi_{(0,1)}\|_p.$$
On the other hand, it follows from the triangle inequality that
$$\|x\|_p\leq\|\mu(x)\chi_{(0,1)}\|_p+\|(\mu(k,x))_{k\geq1}\|_p.$$
Using standard inequalities
$$\sup_{p>1}f(p)+g(p)\leq\sup_{p>1}f(p)+\sup_{p>1}g(p),$$
$$\sup_{p>1}\max\{f(p),g(p)\}\geq\max\{\sup_{p>1}f(p),\sup_{p>1}g(p)\}\geq\frac12\big(\sup_{p>1}f(p)+\sup_{p>1}g(p)\big),$$
we obtain
$$\sup_{1<p\leq 2}(p-1)\|x\|_p\approx\sup_{1<p\leq 2}(p-1)\|\mu(x)\chi_{(0,1)}\|_p+\sup_{1<p\leq 2}(p-1)\|(\mu(k,x))_{k\geq1}\|_p,$$
$$\sup_{2\leq p<\infty}\frac1p\|x\|_p\approx\sup_{2\leq p<\infty}\frac1p\|\mu(x)\chi_{(0,1)}\|_p+\sup_{2\leq p<\infty}\frac1p\|(\mu(k,x))_{k\geq1}\|_p.$$
	
It is immediate that
$$\sup_{1<p\leq 2}(p-1)\|\mu(x)\chi_{(0,1)}\|_p=\|\mu(x)\chi_{(0,1)}\|_2.$$
Applying Lemma \ref{momenta seq} to the sequence $(\mu(k,y))_{k\geq1},$ we obtain
$$\sup_{1<p\leq 2}(p-1)\|(\mu(k,x))_{k\geq1}\|_p\approx\|(\mu(k,x))_{k\geq1}\|_{m_{{\rm log}}}.$$
Substituting these expressions, we obtain the first assertion.

By Lemma \ref{momenta 01}, we have
$$\sup_{2\leq p<\infty}\frac1p\|\mu(x)\chi_{(0,1)}\|_p\approx\|\mu(x)\chi_{(0,1)}\|_{L_{\Phi}(0,1)}.$$
It is immediate that
$$\sup_{2\leq p<\infty}\frac1p\|(\mu(k,x))_{k\geq1}\|_p=\frac12\|(\mu(k,x))_{k\geq1}\|_2.$$
Substituting these expressions, we obtain the second assertion.
\end{proof}

\begin{lem}\label{expl1 lemma} The Marcinkiewicz space $M_{\phi}(0,1)$ coincides with the Orlicz space $L_{\Phi}(0,1),$ where $\Phi(t)=e^t-1,$ $t>0.$
\end{lem}
\begin{proof} Recall that the fundamental function of an Orlicz space $L_{\Phi}$ is given by the formula (see e.g. \cite[page 101]{KPS})
$$t\mapsto\frac1{\Phi^{-1}(t^{-1})},\quad t\in(0,1).$$
For $\Phi(t)=e^t-1,$ $t\in(0,1)$, the fundamental function of $L_{\Phi}(0,1)$ is	$$t\mapsto\frac1{\log(1+\frac1t)},\quad t\in(0,1).$$
On the other hand, the fundamental function of the Marcinkiewicz space $M_{\phi}$ is given by the formula (see e.g. \cite[page 114]{KPS})
$$t\mapsto\frac{t}{\phi(t)},\quad t\in(0,1).$$
For our choice of $\phi,$ the fundamental function of $M_{\phi}(0,1)$ is
$$t\mapsto\frac1{\log(\frac{e^2}{t})},\quad t\in(0,1).$$
Obviously, the fundamental functions of our spaces $L_{\Phi}(0,1)$ and $M_{\phi}(0,1)$ are equivalent. By \cite[Theorem II.5.7]{KPS}, we have $L_{\Phi}(0,1)\subset M_{\phi}(0,1).$
	
On the other hand, for the function
$$\phi':t\mapsto \log(\frac{e}{t}),\quad t\in(0,1),$$
we have that
$$\Phi(\frac12\phi'):t\to (\frac{e}{t})^{\frac12}-1,\quad t\in(0,1),$$
is integrable. In partiular, $\phi'\in L_{\Phi}.$ By the definition of Marcinkiewicz space, we have
$$y\prec\prec \|y\|_{M_{\phi}}\phi',\quad y\in M_{\phi}(0,1).$$
Since every Orlicz space is closed under the Hardy-Littlewood-Polya submajorization, it follows that $M_{\phi}(0,1)\subset L_{\Phi}(0,1).$
	
The result then follows by double inclusion.
\end{proof}

\begin{proof}[Proof of Proposition \ref{momenta proposition}] It is clear that
$$\sup_{1<p<\infty}f(p)=\max\{\sup_{1<p\leq 2}f(p),\sup_{2\leq p<\infty}f(p)\},$$
we write
$$\sup_{1<p<\infty}\min\{p^{-1},p'^{-1}\}\|x\|_p=\max\left\{\sup_{2\leq p<\infty}p^{-1}\|x\|_p,\sup_{1<p\leq 2}p'^{-1}\|x\|_p\right\}.$$
Since
$$(p-1)\leq p'^{-1}\leq 2(p-1),\quad 1<p\leq 2,$$
it follows that
$$\sup_{1<p\leq 2}p'^{-1}\|x\|_p\approx \sup_{1<p\leq 2}(p-1)\|x\|_p\stackrel{C.\ref{momenta  cor}}{\approx}\|\mu(x)\chi_{(0,1)}\|_2+\|(\mu(k,x))_{k\geq1}\|_{m_{{\rm log}}}.$$
On the other hand, we have
$$\sup_{2\leq p<\infty}p^{-1}\|x\|_p\stackrel{C.\ref{momenta cor}}{\approx}\|\mu(x)\chi_{(0,1)}\|_{L_{\Phi}(0,1)}+\|(\mu(k,x))_{k\geq1}\|_2.$$
Substituting these expressions and using Lemma \ref{expl1 lemma}, we obtain
$$\sup_{1<p<\infty}\min\{p^{-1},p'^{-1}\}\|x\|_p\approx\|\mu(x)\chi_{(0,1)}\|_2+\|(\mu(k,x))_{k\geq1}\|_{m_{\phi}}+$$
$$+\|\mu(x)\chi_{(0,1)}\|_{M_{\phi}(0,1)}+\|(\mu(k,x))_{k\geq1}\|_2\approx\|x\|_{M_{\phi}}.$$
\end{proof}

\section{Extrapolation in Pursuit of Optimal Bounds}\label{extrapolation section}
In this section, we state and prove two new type of extrapolation theorems. As key ingredients, they are applied to prove the main results of this paper (see Section \ref{main section}). These extrapolation results are of independent interests and should be compared with the one given in \cite[Theorem 14]{STZ}.

Through the rest of the paper, the following convention is employed.

\begin{conv} Let $X$ and $Y$ be linear spaces, let $(X_i)_{i\in I}\subset X$ be a family of linear subspaces; let $(T_i:X_i\to Y)_{i\in I}$ be a family of linear maps. If
$$T_i|_{X_i\cap X_j}=T_j|_{X_i\cap X_j},\quad i,j\in I,$$
then there exists a linear map $T:\sum_{i\in I}X_i\to Y$ such that $T_i=T|_{X_i}$ for every $i\in I.$ In this case, we simply write $T:X_i\to Y.$
\end{conv}

\subsection{First extrapolation theorem}
Assume that $(\mathcal{N}_1,\nu_1)$ and $(\mathcal{N}_2,\nu_2)$ are semifinite von Neumann algebras. The first extrapolation theorem is based on the following assumption.

\begin{feco}\label{first extrapolation condition} Suppose that $T:L_p(\mathcal{N}_1,\nu_1)\to L_p(\mathcal{N}_2,\nu_2)$ is a linear bounded operator for all $1<p<\infty$ and
$$\|T\|_{L_p\to L_p}\leq \max\{p,p'\},\quad 1<p<\infty,$$
where $p'$ is the conjugate index of $p$.
\end{feco}

Under Condition \ref{first extrapolation condition}, we prove the following extrapolation theorem, which is the main result of this subsection. As a tool, this result allows us to obtain the (DST) and (DMT) inequalities in a convenient way.
Note that our first result, Theorem~\ref{first extrapolation theorem}, implies the classical extrapolation result of Yano \cite{Yano}.

\begin{thm}\label{first extrapolation theorem} If $T$ satisfies First Extrapolation Condition, then
$$Tx\prec\prec c_{{\rm abs}}S\mu(x),\quad x\in \Lambda_{\log}(\mathcal N_1,\nu_1).$$
\end{thm}

\begin{cor}\label{first extrapolation corollary} If $T$ satisfies First Extrapolation Condition and if $\|T\|_{L_1\to L_{1,\infty}}\leq 1,$ then
$$\mu(Tx)\leq c_{{\rm abs}}S\mu(x),\quad x\in \Lambda_{\log}(\mathcal N_1,\nu_1).$$
\end{cor}

\begin{rem}\label{first extrapolation remark} Take $\mathcal{N}_1=\mathcal{N}_2=L_{\infty}(0,\infty)$ (equipped with the usual Lebesgue integral). By Lemma \ref{cesaro-norm}, the operator $T=S$ satisfies First Extrapolation Condition. Obviously, one cannot improve the inequalities
$$Sx\prec\prec c_{{\rm abs}}S\mu(x),\quad \mu(Sx)\prec\prec c_{{\rm abs}}S\mu(x).$$
Thus, results in Theorem \ref{first extrapolation theorem} and Corollary \ref{first extrapolation corollary} are optimal.
\end{rem}

The following lemma describes the behavior of the operator $T$ on  $(L_1\cap L_\infty)(\mathcal{N}_1).$

\begin{lem}\label{first t minimal lemma} If $T$ satisfies First Extrapolation Condition,
then
$$\|T\|_{L_1\cap L_{\infty}\to M_{\phi}}\leq c_{{\rm abs}}.$$
\end{lem}
\begin{proof} First note that $(L_1\cap L_{\infty})(\mathcal{N}_1,\nu_1)\subset L_p(\mathcal{N}_1,\nu_1)$ for all $1<p<\infty.$ In particular, it follows from the Condition \ref{first extrapolation condition} that $Tx$ is well defined for all $x\in (L_1\cap L_{\infty})(\mathcal{N}_1,\nu_1).$ Moreover, we have
$$\min\{p^{-1},p'^{-1}\}\|Tx\|_p\leq \|x\|_p,\quad 1<p<\infty,\quad x\in (L_1\cap L_{\infty})(\mathcal{N}_1,\nu_1).$$
Taking supremum over all $1<p<\infty,$ we obtain
$$\sup_{1<p<\infty}\min\{p^{-1},p'^{-1}\}\|Tx\|_p\leq\sup_{1<p<\infty}\|x\|_p,\quad x\in (L_1\cap L_{\infty})(\mathcal{N}_1,\nu_1).$$
Obviously,
$$\sup_{1<p<\infty}\|x\|_p\approx\|x\|_{L_1\cap L_{\infty}}.$$
By Proposition \ref{momenta proposition}, we have
$$\sup_{1<p<\infty}\min\{p^{-1},p'^{-1}\}\|Tx\|_p\approx\|Tx\|_{M_{\phi}}.$$
Therefore,
$$\|Tx\|_{M_{\phi}}\lesssim \|x\|_{L_1\cap L_{\infty}},\quad x\in (L_1\cap L_{\infty})(\mathcal{N}_1,\nu_1),$$
which is the desired result.
\end{proof}

The following lemma is routine and, therefore, its proof is omitted.

\begin{lem}\label{first duality operator lemma} If $T$ satisfies First Extrapolation Condition, then there exists a linear operator $T^{\ast}:L_p(\mathcal{N}_2,\nu_2)\to L_p(\mathcal{N}_1,\nu_1),$ $1<p<\infty,$ such that
\begin{enumerate}[{\rm (i)}]
\item for every $1<p<\infty,$ we have
$$\|T^{\ast}\|_{L_p\to L_p}\leq \max\{p,p'\}.$$
\item for every $1<p<\infty,$ the operator $T^{\ast}:L_{p'}(\mathcal{N}_2,\nu_2)\to L_{p'}(\mathcal{N}_1,\nu_1)$ is the Banach adjoint of the operator $T:L_p(\mathcal{N}_1,\nu_1)\to L_p(\mathcal{N}_2,\nu_2).$
\end{enumerate}
\end{lem}

The following lemma describes the behavior of the maximal domain of the operator $T.$

\begin{lem}\label{first t maximal lemma} If $T$ satisfies First Extrapolation Condition, then $T$ admits a bounded linear extension $T:\Lambda_{\phi}(\mathcal{N}_1,\nu_1)\to (L_1+L_{\infty})(\mathcal{N}_2,\nu_2).$ Moreover,
$$\|T\|_{\Lambda_{\phi}\to L_1+L_{\infty}}\leq c_{{\rm abs}}.$$
\end{lem}
\begin{proof} Let $x\in L_2(\mathcal{N}_1,\nu_1)$ so that $Tx\in L_2(\mathcal{N}_2,\nu_2).$ By Lemma \ref{general koethe fact} and \eqref{nc koethe eq}, we have
$$\|Tx\|_{L_1+L_{\infty}}=\sup_{\|y\|_{L_1\cap L_{\infty}}\leq 1}|\langle Tx,y\rangle|.$$
By Proposition \ref{M-L-Properties}, the K\"{o}the dual of $\Lambda_\phi(0,\infty)$ is $M_\phi(0,\infty).$ By \eqref{nc koethe eq}, we have
$$\Lambda_\phi(\mathcal{N}_1,\nu_1)^\times =\Lambda_\phi^\times(\mathcal{N}_1,\nu_1)=M_\phi(\mathcal{N}_1,\nu_1).$$
It follows now from the H\"{o}lder inequality (see Lemma \ref{holder}) that
$$|\langle Tx,y\rangle|=|\langle x,T^{\ast}y\rangle|\leq\|x\|_{\Lambda_{\phi}}\|T^{\ast}y\|_{M_{\phi}}$$
for every $x\in L_2(\mathcal{N}_1,\nu_1)$ and for every $y\in (L_1\cap L_{\infty})(\mathcal{N}_2,\nu_2).$ Consequently,
$$\|Tx\|_{L_1+L_{\infty}}\leq\|x\|_{\Lambda_{\phi}}\cdot \sup_{\|y\|_{L_1\cap L_{\infty}}\leq 1}\|T^{\ast}y\|_{M_{\phi}}.$$
By Lemma \ref{first duality operator lemma}, the operator $T^{\ast}$ also satisfies First Extrapolation Condition. Applying Lemma \ref{first t minimal lemma} to the operator $T^{\ast},$ we obtain
$$\sup_{\|y\|_{L_1\cap L_{\infty}}\leq 1}\|T^{\ast}y\|_{M_{\phi}}\leq c_{{\rm abs}},\quad x\in L_2(\mathcal{N}_1,\nu_1).$$
Therefore, we have
$$\|Tx\|_{L_1+L_{\infty}}\leq c_{{\rm abs}}\|x\|_{\Lambda_{\phi}},\quad x\in L_2(\mathcal{N}_1,\nu_1).$$
Since $(L_1\cap L_{\infty})(\mathcal{N}_1,\nu_1)$ is dense in $\Lambda_{\phi}(\mathcal{N}_1,\nu_1)$ (and, hence, so is $L_2(\mathcal{N}_1,\nu_1)$), the assertion follows.
\end{proof}

\begin{lem}\label{valpha direct lemma} If $x\in(\Lambda_{\phi}+L_{\infty})(\mathcal N_1,\nu_1),$ then
$$\int_0^t\mu(s,x)\log\Big(\frac{t}{s}\Big)ds=\int_0^t\left(C\mu(x)\right)(s)ds,\quad \forall t>0.$$
\end{lem}
\begin{proof} Observe that $C\mu(x)\in (L_1+L_{\infty})(0,\infty)$ so that the right hand side is well defined. It follows from a basic calculation that
\begin{align*}\int_0^t(C\mu(x))(s)ds&=\int_0^t\frac1s\int_0^s\mu(u,x)duds\\
&=\int_0^t\mu(u,x)\Big(\int_u^t\frac{ds}{s}\Big)du\\
&=\int_0^t\mu(u,x)\log\Big(\frac{t}{u}\Big)du.
\end{align*}
\end{proof}

In the proof of Theorem \ref{first extrapolation theorem}, we use the following scaling construction.
\begin{construction} Let $(\mathcal{M},\tau)$ be a noncommutative measure space. Consider another noncommutative measure space $(\mathcal{M},t^{-1}\tau).$ It is immediate that  $L_0(\mathcal{M},t^{-1}\tau)=L_0(\mathcal{M},\tau).$ For every $x\in L_0(\mathcal{M},\tau)$ and for all $t>0,$ we have
$$\mu_{(\mathcal{M},t^{-1}\tau)}(s,x)=\inf\big\{\|x(1-p)\|_{\infty}:\ (t^{-1}\tau)(p)\leq s\big\}=$$
$$=\inf\big\{\|x(1-p)\|_{\infty}:\ \tau(p)\leq st\big\}=\mu(st,x),\quad s>0.$$
Thus,
\begin{equation}\label{main scaling eq}
\mu_{(\mathcal{M},t^{-1}\tau)}(x)=\sigma_{\frac1t}\mu(x).
\end{equation}
\end{construction}

\medskip
Now we are ready to prove Theorem \ref{first extrapolation theorem} and Corollary \ref{first extrapolation corollary}.
\smallskip

\begin{proof}[Proof of Theorem \ref{first extrapolation theorem}]
Let $x\in\Lambda_{\log}(\mathcal N_1,\nu_1)$ and fix $t>0$. We scale the trace in our algebra; namely, instead of the algebras $(\mathcal{N}_1,\nu_1),$ $(\mathcal{N}_2,\nu_2),$ we consider the algebras $(\mathcal{N}_1,t^{-1}\nu_1), (\mathcal{N}_2,t^{-1}\nu_2).$
Obviously, we still have
\begin{align*}
\|Tx\|_{L_p(\mathcal{N}_2,t^{-1}\nu_2)}&=t^{-\frac1p}\|Tx\|_{L_p(\mathcal{N}_2,\nu_2)}\\
&\leq t^{-\frac1p} \max\{p,p'\}\|x\|_{L_p(\mathcal{N}_1,\nu_1)}\\
&= \max\{p,p'\}\|x\|_{L_p(\mathcal{N}_1,t^{-1}\nu_1)}.
\end{align*}
Thus, applying Lemma \ref{first t maximal lemma}, we have
$$\|Tx\|_{(L_1+L_{\infty})(\mathcal{N}_2,t^{-1}\nu_2)}\leq c_{{\rm abs}}\|x\|_{\Lambda_{\phi}(\mathcal{N}_1,t^{-1}\nu_1)}.$$
We have
$$\|Tx\|_{(L_1+L_{\infty})(\mathcal{N}_2,t^{-1}\nu_2)}\stackrel{\eqref{main scaling eq}}{=}\|\sigma_{\frac1t}\mu(x)\|_{L_1+L_{\infty}}=\int_0^1\mu(st,x)ds=(C\mu(Tx))(t).$$
On the other hand, we have
$$\|x\|_{\Lambda_{\phi}(\mathcal{N}_1,t^{-1}\nu_1)}\stackrel{\eqref{main scaling eq}}{=}\|\sigma_{\frac1t}\mu(x)\|_{\Lambda_{\phi}}=\int_0^{\infty}\mu(ut,x)\phi'(u)du=$$
$$=\int_0^1\mu(ut,x)\log(\frac{e}{u})du+\int_1^{\infty}\mu(ut,x)\frac{du}{u}=$$
$$=\frac1t\int_0^t\mu(s,x)\log(\frac{et}{s})ds+\int_t^{\infty}\mu(s,x)\frac{ds}{s}=$$
$$\stackrel{L.\ref{valpha direct lemma}}{=}(C^2\mu(x))(t)+(C^{\ast}\mu(x))(t).$$
A combination of the three last equations yields
$$C\mu(Tx)\leq c_{{\rm abs}}\cdot\big(C^2\mu(x)+C^{\ast}\mu(x)\big).$$
Since $Cy\geq y$ for every decreasing function $y$ and since $C^{\ast}\mu(x)$ is decreasing, it follows that
$$C\mu(Tx)\leq c_{{\rm abs}}\cdot\big(C^2\mu(x)+CC^{\ast}\mu(x)\big)=c_{{\rm abs}}CS\mu(x).$$
The last inequality is equivalent to the assertion of Theorem \ref{first extrapolation theorem}.
\end{proof}

\begin{proof}[Proof of Corollary \ref{first extrapolation corollary}] Let $0\leq x\in \Lambda_{\log}(\mathcal N_1, \nu_1)$  and fix $t>0.$ Set
$$x_1=(x-\mu(t,x))_+\quad \mbox{and }\quad x_2=\min\{x,\mu(t,x)\}.$$
Then by the inequality \eqref{singular-triangle}, we have
$$\mu(2t,Tx)=\mu(2t,Tx_1+Tx_2)\leq \mu(t,Tx_1)+\mu(t,Tx_2).$$
By the assumption, we have $\|T\|_{L_1\to L_{1,\infty}}\leq 1.$ It follows that
$$\mu(t,Tx_1)\leq t^{-1}\|x_1\|_1=t^{-1}\int_0^t\left(\mu(s,x)-\mu(t,x)\right)ds\leq \left(C\mu(x)\right)(t).$$
By Theorem \ref{first extrapolation theorem}, we have
$$\mu(t,Tx_2)\leq (C\mu(Tx_2))(t)\leq c_{{\rm abs}}(CS\mu(x_2))(t).$$
Obviously,
$$(S\mu(x_2))(s)=\frac1s\int_0^s\mu(u,x_2)du+\int_s^{\infty}\mu(u,x_2)\frac{du}{u}=$$
$$=\frac1s\int_0^s\mu(t,x)du+\int_s^t\mu(t,x)\frac{du}{u}+\int_t^{\infty}\mu(u,x)\frac{du}{u}=$$
$$=\mu(t,x)\cdot\log(\frac{et}{s})+(C^{\ast}\mu(x))(t),\quad 0<s<t.$$
Therefore,
$$(CS\mu(x_2))(t)=\frac1t\int_0^t\big(\mu(t,x)\cdot\log(\frac{et}{s})+(C^{\ast}\mu(x))(t)\big)ds=$$
$$=(C^{\ast}\mu(x))(t)+\frac1t\int_0^t\log(\frac{et}{s})ds\cdot\mu(t,x)=(C^{\ast}\mu(x))(t)+2\mu(t,x).$$
Combining the estimates above, we obtain
$$\mu(2t,Tx)\leq (C\mu(x))(t)+c_{{\rm abs}}\cdot((C^{\ast}\mu(x))(t)+2\mu(t,x)).$$
Obviously,
$$(C\mu(x))(t)\leq (S\mu(x))(t),\quad (C^{\ast}\mu(x))(t)\leq (S\mu(x))(t),\quad \mu(t,x)\leq (S\mu(x))(t).$$
It follows that
$$\mu(2t,Tx)\leq c_{{\rm abs}}(S\mu(x))(t)\leq c_{{\rm abs}}(S\mu(x))(2t).$$
Since $t>0$ is arbitrary, the assertion follows.
\end{proof}

\subsection{Second extrapolation theorem}
We now turn to the second extrapolation theorem. Let $(\mathcal{N}_1,\nu_1)$ and $(\mathcal{N}_2,\nu_2)$ be semifinite von Neumann algebras. The second extrapolation result is based on the following condition.

\begin{seco} Suppose that $T:L_p(\mathcal{N}_1,\nu_1)\to L_p(\mathcal{N}_2,\nu_2)$ is a bounded linear operator for all $2\leq p<\infty$ and
$$\|T\|_{L_p\to L_p}\leq p,\quad 2\leq p<\infty.$$
\end{seco}

Now, the principal result of this subsection---the second extrapolation theorem, is stated below. We will make essential use of this result in the proof of (DDD), (Upper$-$DBG), (Lower$-$DBG) inequalities.

\begin{thm}\label{second extrapolation theorem} If $T$ satisfies Second Extrapolation Condition, then
$$\mu^2(Tx)\prec\prec c_{{\rm abs}}(C^{\ast}\mu(x))^2, \quad x\in (\Lambda_{\log}+L_2)(\mathcal{N}_1,\nu_1).$$
\end{thm}

\begin{rem} Take $\mathcal{N}_1=\mathcal{N}_2=L_{\infty}(0,\infty)$ (equipped with the usual Lebesgue integral). By Lemma \ref{cesaro-norm}, the operator $T=C^{\ast}$ satisfies the Second Extrapolation Condition. Obviously, one cannot improve the inequality
$$\mu^2(C^{\ast}x)\prec\prec c_{{\rm abs}}(C^{\ast}\mu(x))^2.$$
Thus, result in Theorem \ref{second extrapolation theorem} is optimal.
\end{rem}

The following lemma is analogous to Lemma \ref{first duality operator lemma}.

\begin{lem}\label{second duality operator lemma} If $T$ satisfies Second Extrapolation Condition, then there exists a linear operator $T^{\ast}:L_p(\mathcal{N}_2,\nu_2)\to L_p(\mathcal{N}_1,\nu_1),$ $1<p\leq 2,$ such that
\begin{enumerate}[{\rm (i)}]
\item for every $1<p\leq 2,$ we have
$$\|T^{\ast}\|_{L_p\to L_p}\leq p'.$$
\item for every $1<p\leq 2,$ the operator $T^{\ast}:L_p(\mathcal{N}_2,\nu_2)\to L_p(\mathcal{N}_1,\nu_1)$ is the Banach adjoint of the operator $T:L_{p'}(\mathcal{N}_1,\nu_1)\to L_{p'}(\mathcal{N}_2,\nu_2).$
\end{enumerate}
\end{lem}

The following lemma describes the behavior of the operator $T^{\ast}$ on the domain $(L_1\cap L_2)(\mathcal{N}_2,\nu_2)$. It should be compared with Lemma \ref{first t minimal lemma}.

\begin{lem}\label{second t minimal lemma} If $T$ satisfies Second Extrapolation Condition, then
$$T^{\ast}:(L_1\cap L_2)(\mathcal{N}_2,\nu_2)\to (M_{{\rm log}}+(L_1\cap L_2))(\mathcal{N}_1,\nu_1).$$
Moreover, we have
$$\|T^{\ast}\|_{L_1\cap L_2\to M_{{\rm log}}+(L_1\cap L_2)}\leq c_{{\rm abs}}.$$
\end{lem}
\begin{proof} Take $x\in (L_1\cap L_2)(\mathcal{N}_2,\nu_2)$ and note that $x\in L_p(\mathcal{N}_2,\nu_2)$ for all $1<p\leq 2.$ In particular, $T^*x$ is well defined by Lemma \ref{second duality operator lemma}. Moreover, we have
$$(p-1)\|T^{\ast}x\|_p\leq p\|x\|_p\leq 2\|x\|_p,\quad 1<p\leq 2.$$
Taking the supremum over $1<p\leq 2,$ we obtain
$$\sup_{1<p\leq 2}(p-1)\|T^{\ast}x\|_p\leq2\sup_{1<p\leq 2}\|x\|_p.$$
Clearly,
$$\sup_{1<p\leq 2}\|x\|_p\approx\|x\|_{L_1\cap L_2}.$$
By Corollary \ref{momenta cor},
$$\sup_{1<p\leq 2}(p-1)\|T^{\ast}x\|_p\approx\|T^{\ast}x\|_{M_{{\rm log}}+(L_1\cap L_2)}.$$
Combining the last $3$ equations, we arrive at
$$\|T^{\ast}x\|_{M_{{\rm log}}+(L_1\cap L_2)}\lesssim\|x\|_{L_1\cap L_2},\quad x\in(L_1\cap L_2)(\mathcal{N}_2,\nu_2).$$
\end{proof}

The following lemma describes the maximal domain of the operator $T$ (it should be compared with Lemma \ref{first t maximal lemma}).

\begin{lem}\label{second t maximal lemma} If $T$ satisfies Second Extrapolation Condition, then
$$T:(\Lambda_{{\rm log}}\cap (L_2+L_{\infty}))(\mathcal{N}_1,\nu_1)\to (L_2+ L_{\infty})(\mathcal{N}_2,\nu_2).$$
Moreover, we have
$$\|T\|_{\Lambda_{{\rm log}}\cap (L_2+L_{\infty})\to L_2+ L_{\infty}}\leq c_{{\rm abs}}.$$
\end{lem}
\begin{proof} Let $x\in L_2(\mathcal{N}_1,\nu_1)$ so that $Tx\in L_2(\mathcal{N}_2,\nu_2).$ By Lemma \ref{general koethe fact} and \eqref{nc koethe eq}, we have
$$\|Tx\|_{L_2+L_{\infty}}=\sup_{\|y\|_{L_1\cap L_2}\leq 1}|\langle Tx,y\rangle|.$$
We infer from Lemma \ref{general koethe fact}, Proposition \ref{M-L-Properties} and \eqref{nc koethe eq} that
$$(\Lambda_{{\rm log}}(\mathcal{N}_1,\nu_1)\cap (L_2+L_{\infty})(\mathcal{N}_1,\nu_1))^\times=$$
$$=\big(\Lambda_{{\rm log}}^\times+ (L_2+L_{\infty})^\times\big)(\mathcal{N}_1,\nu_1)=\big(M_{{\rm log}}+(L_1\cap L_2)\big)(\mathcal{N}_1,\nu_1).$$
Applying the H\"{o}lder inequality (see Lemma \ref{holder}), we get
$$|\langle Tx,y\rangle|=|\langle x,T^{\ast}y\rangle|\leq\|x\|_{\Lambda_{{\rm log}}\cap (L_2+L_{\infty})}\|T^{\ast}y\|_{M_{{\rm log}}+(L_1\cap L_2)}$$
for every $x\in L_2(\mathcal{N}_1,\nu_1)$ and for every $y\in (L_1\cap L_2)(\mathcal{N}_2,\nu_2).$ Thus,
$$\|Tx\|_{L_2+L_{\infty}}\leq\|x\|_{\Lambda_{{\rm log}}\cap (L_2+L_{\infty})}\cdot \sup_{\|y\|_{L_1\cap L_2}\leq 1}\|T^{\ast}y\|_{M_{{\rm log}}+(L_1\cap L_2)}.$$
Using Lemma \ref{second t minimal lemma}, we infer that
$$\|Tx\|_{L_2+L_{\infty}}\leq c_{\rm abs}\|x\|_{\Lambda_{{\rm log}}\cap (L_2+L_{\infty})},\quad x\in L_2(\mathcal{N}_1,\nu_1).$$
Since $L_2(\mathcal{N}_1,\nu_1)$ is dense in $(\Lambda_{{\rm log}}\cap (L_2+L_{\infty}))(\mathcal{N}_1,\nu_1),$ the assertion follows.
\end{proof}

The purpose of the following lemma is similar to that of Lemma \ref{valpha direct lemma}.

\begin{lem}\label{valpha inverse estimate} Let $x\in (L_1+L_{\infty})(\mathcal{N}_1,\nu_1).$ We have
$$C\mu^2(x)+(C^{\ast}\mu(x))^2\leq 2C\big((C^{\ast}\mu(x))^2\big).$$	
\end{lem}
\begin{proof} By definition of $C^{\ast},$ we have
$$\int_0^t(C^{\ast}\mu(x))(s)ds=\int_0^t\int_s^{\infty}\mu(u,x)\frac{du}{u}ds=$$
$$=\int_0^{\infty}\mu(u,x)\Big(\int_0^{\min\{t,u\}}ds\Big)\frac{du}{u}\geq\int_0^t\mu(u,x)du.$$
Since $t>0$ is arbitrary, it follows that $\mu(x)\prec\prec C^{\ast}\mu(x).$

Note that $g\prec\prec f$ always implies $g^2\prec\prec f^2.$ From the preceding paragraph, we obtain
$$\mu^2(x)\prec\prec (C^{\ast}\mu(x))^2.$$
Equivalently,
$$C\mu^2(x)\leq C\big((C^{\ast}\mu(x))^2\big).$$

Since $Cy\geq y$ for every decreasing function $y$ and since $C^{\ast}\mu(x)$ is decreasing, it follows that
$$(C^{\ast}\mu(x))^2\leq C\big(C^{\ast}\mu(x))^2\big).$$
The assertion follows by combining the last $2$ estimates.
\end{proof}

We are now ready to prove Theorem \ref{second extrapolation theorem}.

\begin{proof}[Proof of Theorem \ref{second extrapolation theorem}] Scaling argument is similar to the one used in Theorem \ref{first extrapolation theorem}. We include the details for convenience of the reader. Fix $t>0$. Instead of the algebras $(\mathcal{N}_1,\nu_1),$ $(\mathcal{N}_2,\nu_2),$ consider the algebras $(\mathcal{N}_1,t^{-1}\nu_1).$ $(\mathcal{N}_2,t^{-1}\nu_2).$ We then have
\begin{equation*}
\|Tx\|_{L_p(\mathcal{N}_2,t^{-1}\nu_2)}=t^{-\frac1p}\|Tx\|_{L_p(\mathcal{N}_2,\nu_2)}\leq t^{-\frac1p}p\|x\|_{L_p}= p\|x\|_{L_p(\mathcal{N}_1,t^{-1}\nu_1)}.
\end{equation*}
Therefore, by Lemma \ref{second t maximal lemma}, we have
\begin{equation}\label{set eq1}
\|Tx\|_{(L_2+L_{\infty})(\mathcal{N}_2,t^{-1}\nu_2)}\leq c_{\rm abs}\|x\|_{(\Lambda_{{\rm log}}\cap (L_2+L_{\infty}))(\mathcal{N}_1,t^{-1}\nu_1)}.
\end{equation}

We have
$$\|Tx\|_{L_2+L_{\infty}(\mathcal{N}_2,t^{-1}\nu_2)}\stackrel{\eqref{main scaling eq}}{=}\|\sigma_{\frac1t}\mu(Tx)\|_{L_2+L_{\infty}}.$$
By Lemma \ref{lem-L1-Linfty}, we have
$$\|z\|_{L_2+L_{\infty}}^2\approx \int_0^1\mu^2(u,z)du.$$
Thus,
\begin{equation}\label{set eq2}
\|Tx\|_{L_2+L_{\infty}(\mathcal{N}_2,t^{-1}\nu_2)}^2\approx \int_0^1\mu^2(tu,z)du=\frac1t\int_0^t\mu^2(s,x)ds.
\end{equation}

Similarly,
$$\|x\|_{(\Lambda_{{\rm log}}\cap (L_2+L_{\infty}))(\mathcal{N}_1,t^{-1}\nu_1)}=\|\sigma_{\frac1t}\mu(x)\|_{\Lambda_{{\rm log}}\cap (L_2+L_{\infty})}^2.$$
We have
$$\|z\|_{\Lambda_{{\rm log}}\cap (L_2+L_{\infty})}^2\approx\int_0^1\mu^2(u,z)du+\big(\int_1^{\infty}\mu(u,z)\frac{du}{u}\big)^2.$$
Thus,
\begin{multline}\label{set eq3}
\|x\|_{(\Lambda_{{\rm log}}\cap (L_2+L_{\infty}))(\mathcal{N}_1,t^{-1}\nu_1)}^2\approx \\
\approx \int_0^1\mu^2(tu,z)du+\big(\int_1^{\infty}\mu(tu,z)\frac{du}{u}\big)^2=\\
=\frac1t\int_0^t\mu^2(s,z)ds+\big(\int_t^{\infty}\mu(s,z)\frac{ds}{s}\big)^2.
\end{multline}

Substituting \eqref{set eq1} and \eqref{set eq2} into \eqref{set eq3}, we obtain
$$\frac1t\int_0^t\mu^2(s,Tx)ds\lesssim\frac1t\int_0^t\mu^2(s,x)ds+\Big(\int_t^{\infty}\mu(s,x)\frac{ds}{s}\Big)^2,\quad t>0.$$
Since $t>0$ is arbitrary, it follows that
$$C\mu^2(Tx)\lesssim C\mu^2(x)+(C^{\ast}\mu(x))^2.$$
It follows now from Lemma \ref{valpha inverse estimate} that
$$C\mu^2(Tx)\lesssim C\big((C^{\ast}\mu(x))^2\big).$$
The last inequality is equivalent to the assertion of Theorem \ref{first extrapolation theorem}.
\end{proof}

\section{Proof of the main results}\label{main section}

In this section, we provide the proof of the distributional Stein, dualised Doob, Burkholder-Gundy inequalities, and also the distributional estimate for martingale transforms. We start with the following lemma, which is a simple exercise on convergence in the strong operator topology and convergence in $L_p.$ We provide the proof for the convenience of the reader.
\begin{lem}\label{convergence fact} Let $\mathcal{M}\subset B(H)$ be a semifinite von Neumann algebra.
\begin{enumerate}[{\rm (i)}]
\item If $(b_k)_{k\geq0}\subset L_{\infty}(\mathcal{M})$ are such that $\sum_{k\geq0}|b_k|^2\in L_{\infty}(\mathcal{M}),$ then the series $\sum_{k\geq0}b_k\otimes e_{k0}$ converges in the strong operator topology in $L_{\infty}(\mathcal{M}\overline\otimes B(\ell_2)).$
\item If $(b_k)_{k\geq0}\subset L_p(\mathcal{M}),$ $0<p<\infty,$ are such that $\sum_{k\geq0}|b_k|^2\in L_{\frac{p}{2}}(\mathcal{M}),$ then the series $\sum_{k\geq0}b_k\otimes e_{k0}$ converges in $L_p(\mathcal{M}\overline\otimes B(\ell_2)).$
\end{enumerate}
\end{lem}
\begin{proof} By Vigier's theorem \cite[Theorem 4.1.1]{Murphy90}, the series $\sum_{k\geq0}|b_k|^2$ converges in $B(H)$ in strong operator topology. Fix $\xi\in H\overline{\otimes}\ell_2$ and set $\eta\otimes e_0=(1\otimes e_{00})\xi.$ We claim that the sequence $\big((\sum_{k=0}^nb_k\otimes e_{k0})\xi\big)_{n\geq0}$ is Cauchy in $H\overline{\otimes}\ell_2.$ To see this, fix $\varepsilon>0$ and choose $n\in\mathbb{Z}_+$ such that
$$\sum_{k=n+1}^{\infty}\langle |b_k|^2\eta,\eta\rangle_H<\varepsilon.$$
For $m>n,$ we have
\begin{align*}
\Big\|&\Big(\sum_{k=0}^nb_k\otimes e_{k0}\Big)\xi-\Big(\sum_{k=0}^mb_k\otimes e_{k0}\Big)\xi\Big\|_{H\overline{\otimes}\ell_2}^2\\
&=\Big\|\Big(\sum_{k=n+1}^m|b_k|^2\Big)^{\frac12}\eta\Big\|_H^2
=\Big\langle\Big(\sum_{k=n+1}^m|b_k|^2\Big)\eta,\eta\Big\rangle_H\\
&\leq\sum_{k=n+1}^{\infty}\langle |b_k|^2\eta,\eta\rangle_H<\varepsilon.
\end{align*}
This proves the claim. By the preceding arguments, the sequence  $\big(\sum_{k=0}^nb_k\otimes e_{k0}\big)_{n\geq0}$ is strongly Cauchy. By completeness, it is strongly convergent.
	
To see the second assertion, note that for $0<p<\infty$, we have
\begin{align*}
\big\|\sum_{k=n+1}^mb_k\otimes e_{k0}\big\|_p&=\big\|\big(\sum_{k=n+1}^m|b_k|^2\big)^{\frac12}\big\|_p\\
&=\big\|\sum_{k=n+1}^m|b_k|^2\big\|_{\frac{p}{2}}^{\frac12}\leq\big\|\sum_{k=n+1}^{\infty}|b_k|^2\big\|_{\frac{p}{2}}^{\frac{1}{2}}.
\end{align*}
The sequence $\big(\sum_{k=n+1}^{\infty}|b_k|^2\big)_{n\geq0}\subset L_{p/2}(\mathcal{M})$ converges to $0$ in order (see page 5 for the definition of convergence in order). Since the (quasi-)norm in $L_{\frac{p}{2}}(\mathcal{M})$ is order-continuous, it follows that the latter sequence also converges to $0$ in (quasi-)norm.  Therefore, the sequence
$$\big(\sum_{k=0}^nb_k\otimes e_{k0}\big)_{n\geq0}\subset L_p(\mathcal{M}\overline{\otimes}B(\ell_2))$$
is Cauchy. By completeness, it converges in $L_p(\mathcal{M}\overline{\otimes}B(\ell_2)).$
\end{proof}

\begin{lem}\label{lp tensor fact} If $x\in L_p(\mathcal{M}\bar{\otimes}B(\ell_2)),$ then there exists a unique sequence $(x_{kl})_{k,l\geq0}\subset L_p(\mathcal{M})$ such that
$$x=\lim_{N\to\infty}\sum_{k,l=0}^Nx_{kl}\otimes e_{kl}$$
in $L_p(\mathcal{M}\bar{\otimes}B(\ell_2)).$
\end{lem}
\begin{proof} Set
$$p_N=\sum_{k=0}^N1\otimes e_{kk}\mbox{ so that }p_N\uparrow 1\mbox{ as }N\to\infty.$$
We have
$$\|x-p_Nxp_N\|_p\leq\|(1-p_N)x\|_p+\|x(1-p_N)\|_p=$$
$$=\big\||x^{\ast}|(1-p_N)|x^{\ast}|\big\|_{\frac{p}{2}}^{\frac12}+\big\||x|(1-p_N)|x|\big\|_{\frac{p}{2}}^{\frac12}.$$
Since
$$|x^{\ast}|(1-p_N)|x^{\ast}|\downarrow 0,\quad |x|(1-p_N)|x|\downarrow 0,$$
it follows from the order-continuity of the norm in $L_{\frac{p}{2}}$ that
$$\big\||x^{\ast}|(1-p_N)|x^{\ast}|\big\|_{\frac{p}{2}}\downarrow0,\quad \big\||x|(1-p_N)|x|\big\|_{\frac{p}{2}}\downarrow 0.$$
Hence,
$$\|x-p_Nxp_N\|_p\to 0,\quad N\to\infty.$$
	
If $X\in L_{\infty}(\mathcal{M}\bar{\otimes}B(\ell_2))$ is such that
$$X=(1\otimes e_{00})\cdot X\cdot (1\otimes e_{00}),$$
then there exists a net $(X_i)_{i\in I}$ of elementary tensors such that $X_i\to X$ in weak operator topology. Thus,
$$(1\otimes e_{00})\cdot X_i\cdot (1\otimes e_{00})\to (1\otimes e_{00})\cdot X\cdot (1\otimes e_{00})=X$$
in weak operator topology. On the other hand, we have
$$(1\otimes e_{00})\cdot X_i\cdot (1\otimes e_{00})=x_i\otimes e_{00}\mbox{ for some }x_i\in\mathcal{M}.$$
Thus, $x_i\otimes e_{00}\to X$ in weak operator topology. Let $x$ be a cluster point of the net $(x_i)_{i\in I}.$ It follows that $X=x\otimes e_{00}$ as stated.
	
If $X\in L_p(\mathcal{M}\bar{\otimes}B(\ell_2))$ is such that
$$X=(1\otimes e_{00})\cdot X\cdot (1\otimes e_{00}),$$
then there exists $x\in L_p(\mathcal{M})$ such that $X=x\otimes e_{00}.$ Indeed, without loss of generality, $X\geq0.$ For every $t>0,$ we have
$$\chi_{(t,\infty)}(X)\leq\chi_{(0,\infty)}(X)\leq (1\otimes e_{00})$$
and, therefore,
$$\chi_{(t,\infty)}(X)=(1\otimes e_{00})\cdot\chi_{(t,\infty)}(X)\cdot (1\otimes e_{00}).$$
By the preceding paragraph, there exists an element $x_t\in\mathcal{M}$ such that
$$\chi_{(t,\infty)}(X)=x_t\otimes e_{00},\quad t>0.$$
It is immediate that $x_t^2=x_t^{\ast}=x_t$ and, therefore, $x_t$ is a projection $p_t\in P(\mathcal{M}).$ It is immediate that $(p_t)_{t>0}$ forms a spectral family of a positive operator $x$ affiliated with $\mathcal{M}.$ By spectral theorem, we have $X=x\otimes e_{00}.$ Thus, $x\in L_p(\mathcal{M})$ as required.

Let
$$X_{kl}=(1\otimes e_{0k})\cdot x\cdot (1\otimes e_{l0}),\quad k,l\in\geq0.$$
Obviously,
$$X_{kl}\in L_p(\mathcal{M}\bar{\otimes}B(\ell_2)),\quad X_{kl}=(1\otimes e_{00})\cdot X_{kl}\cdot (1\otimes e_{00}).$$
By the preceding paragraph, there exists $x_{kl}\in L_p(\mathcal{M})$ such that $X_{kl}=x_{kl}\otimes e_{00}.$ Therefore,
$$(1\otimes e_{kk})\cdot x\cdot (1\otimes e_{ll})=x_{kl}\otimes e_{kl}.$$
We now have
$$\sum_{k,l=0}^Nx_{kl}\otimes e_{kl}=\sum_{k,l=0}^N(1\otimes e_{kk})\cdot x\cdot (1\otimes e_{ll})=p_Nxp_N.$$
The existence claim follows now from the first paragraph.

Now, if $(y_{kl})_{k,l\geq0}\subset L_p(\mathcal{M})$ is such that
$$\sum_{k,l=0}^Ny_{kl}\otimes e_{kl}\to x,$$
then
$$(1\otimes e_{m_1m_1})\cdot \big(\sum_{k,l=0}^Ny_{kl}\otimes e_{kl}\big)\cdot (1\otimes e_{m_2m_2})\to (1\otimes e_{m_1m_1})\cdot x\cdot (1\otimes e_{m_2m_2}).$$
For $N\geq\max\{m_1,m_2\},$ we have
$$(1\otimes e_{m_1m_1})\cdot \big(\sum_{k,l=0}^Ny_{kl}\otimes e_{kl}\big)\cdot (1\otimes e_{m_2m_2})=y_{m_1m_2}\otimes e_{m_1m_2}.$$
Thus,
$$y_{m_1m_2}\otimes e_{m_1m_2}=(1\otimes e_{m_1m_1})\cdot x\cdot (1\otimes e_{m_2m_2})=x_{m_1m_2}\otimes e_{m_1m_2}.$$
The uniqueness claim now follows.
\end{proof}

\begin{cor}\label{lp tensor cor} If $x\in L_p(\mathcal{M}\bar{\otimes}B(\ell_2)),$ $0<p<\infty,$ and if $(x_{kl})_{k,l\geq0}$ is the sequence constructed in Lemma \ref{lp tensor fact}, then
$$\sum_{k\geq0}|x_{k0}|^2\in L_{\frac{p}{2}}(\mathcal{M})\mbox{ and }\Big\|\Big(\sum_{k\geq0}|x_{k0}|^2\Big)^{\frac12}\Big\|_p\leq\|x\|_p.$$
\end{cor}
\begin{proof} We have
$$\sum_{k=0}^Nx_{k0}\otimes e_{k0}=(\sum_{k=0}^N1\otimes e_{kk})\cdot x\cdot (1\otimes e_{00}).$$
Thus,
$$\|\sum_{k=0}^Nx_{k0}\otimes e_{k0}\|_p\leq\|\sum_{k=0}^N1\otimes e_{kk}\|_{\infty}\|x\|_p\|1\otimes e_{00}\|_{\infty}=\|x\|_p.$$
Further, a direct computation yields
$$\Big|\sum_{k=0}^Nx_{k0}\otimes e_{k0}\Big|^2=\sum_{k=0}^N|x_{k0}|^2\otimes e_{00}=(\sum_{k=0}^N|x_{k0}|^2)\otimes e_{00}.$$
Thus,
$$\|\sum_{k=0}^N|x_{k0}|^2\|_{\frac{p}{2}}=\|\sum_{k=0}^Nx_{k0}\otimes e_{k0}\|_p^2\leq\|x\|_p^2.$$
The assertion follows now from the Fatou property of the space $L_p(\mathcal{M}\bar{\otimes}B(\ell_2)).$
\end{proof}

\begin{proof}[Proof of Theorem \ref{stein thm}] We use First Extrapolation Theorem to show the distributional Stein inequality (DST). Consider the mapping $T$ which is formally given by the formula
\begin{equation}\label{stein t def}
T:x\to \sum_{k\geq0}\mathcal{E}_{k}x_{k0}\otimes e_{k0},\quad x\in L_p(\mathcal{M}\bar{\otimes}B(\ell_2)),
\end{equation}
where the sequence $(x_{kl})_{k,l\geq0}$ is constructed in Lemma \ref{lp tensor fact}. We claim that, for every  $x\in L_p(\mathcal{M}\bar{\otimes}B(\ell_2)),$ $1<p<\infty,$ the series on the right hand side in \eqref{stein t def} converges in $L_p(\mathcal{M}\bar{\otimes}B(\ell_2));$ moreover, the operator  $$T:L_p(\mathcal{M}\bar{\otimes}B(\ell_2))\to L_p(\mathcal{M}\bar{\otimes}B(\ell_2)),\quad 1< p<\infty,$$
is a well defined bounded mapping which satisfies  First Extrapolation Condition.

To see the claim, using (ST) and Corollary \ref{lp tensor cor}, we obtain
$$\Big\|\Big(\sum_{k\geq0}|\mathcal{E}_{k}x_{k0}|^2\Big)^{\frac12}\Big\|_p\leq c_{{\rm abs}}\max\{p,p'\} \cdot \Big\|\Big(\sum_{k\geq0}|x_{k0}|^2\Big)^{\frac12}\Big\|_p\leq c_{{\rm abs}}\max\{p,p'\} \cdot\|x\|_p.$$
By Lemma \ref{convergence fact}, the series on the right hand side in \eqref{stein t def} converges in $L_p(\mathcal{M}\overline{\otimes}B(\ell_2)).$  In particular, $T$ is well defined. Moreover, for every $1<p<\infty,$ we have
$$\|Tx\|_p\leq c_{{\rm abs}}\max\{p,p'\} \cdot\|x\|_p,\quad x\in L_p(\mathcal{M}\bar{\otimes}B(\ell_2)).$$
In other words, the operator $T$ satisfies First Extrapolation Condition.

In addition, by \cite[Theorem 3.10]{Randrianantoanina}, for every $x\in L_1(\mathcal{M}\overline{\otimes}B(\ell_2)),$
$$\Big\|\Big(\sum_{k\geq0}|\mathcal{E}_{k}x_{k0}|^2\Big)^{\frac12}\Big\|_{1,\infty}\leq c_{{\rm abs}} \Big\|\Big(\sum_{k\geq0}|x_{k0}|^2\Big)^{\frac12}\Big\|_1\leq c_{{\rm abs}} \cdot\|x\|_1.$$
Therefore,
$$\|Tx\|_{1,\infty}\leq c_{{\rm abs}}\|x\|_1.$$
Hence, Corollary \ref{first extrapolation corollary} is applicable to the operator $T.$

Applying Corollary \ref{first extrapolation corollary} to the operator $T$, we obtain
$$\mu(Tx)\leq c_{{\rm abs}}S\mu(x),\quad x\in\Lambda_{\log}(\mathcal{M}\bar{\otimes}B(\ell_2)).$$
Now taking
$$x=\sum_{k\geq 0}x_k\otimes e_{k0}\in \Lambda_{\log}(\mathcal{M}\bar{\otimes}B(\ell_2)),$$
we have
$$\mu\Big(\sum_{k\geq 0}\mathcal{E}_kx_k\otimes e_{k0}\Big)\leq c_{{\rm abs}}S\mu\Big(\sum_{k\geq 0}x_k\otimes e_{k0}\Big),$$
or, equivalently,
\begin{equation}\label{stein main eq}
\mu^{\frac12}\Big(\Big|\sum_{k\geq 0}\mathcal{E}_kx_k\otimes e_{k0}\Big|^2\Big)\leq c_{{\rm abs}}S\mu^{\frac12}\Big(\Big|\sum_{k\geq 0}x_k\otimes e_{k0}\Big|^2\Big).
\end{equation}

It is immediate that
$$\Big|\sum_{k\geq 0}x_k\otimes e_{k0}\Big|^2=\sum_{k,l\geq0}x_k^{\ast}x_l\otimes e_{0k}e_{l0}=\sum_{k\geq0}x_k^{\ast}x_k\otimes e_{00}=\big(\sum_{k\geq0}x_k^{\ast}x_k\big)\otimes e_{00}$$
and
\begin{align*}
\Big|\sum_{k\geq 0}\mathcal{E}_kx_k\otimes e_{k0}\Big|^2&=\sum_{k,l\geq0}(\mathcal{E}_kx_k)^{\ast}(\mathcal{E}_lx_l)\otimes e_{0k}e_{l0}\\
&=\sum_{k\geq0}(\mathcal{E}_kx_k)^{\ast}(\mathcal{E}_kx_k)\otimes e_{00}=\big(\sum_{k\geq0}|\mathcal{E}_kx_k|^2\big)\otimes e_{00}.
\end{align*}
Substituting these expressions into \eqref{stein main eq}, we obtain
$$\mu^{\frac12}\Big(\sum_{k\geq 0}|\mathcal{E}_kx_k|^2\Big)\leq c_{{\rm abs}}S\mu^{\frac12}\Big(\sum_{k\geq 0}|x_k|^2\Big).$$
This completes the proof.
\end{proof}

In preparation for the proof of Theorem \ref{dual doob thm}, we recall a crucial result established by Junge in \cite[Proposition 2.8]{Junge-doob} (see also the exposition in \cite{JSZZ}, Appendix, Proof of Theorem 5.6).

\begin{prop}\label{junge expectation theorem} Let $\mathcal{M}$ be a semifinite von Neumann algebra and $\mathcal{N}$ be a semifinite von Neumann subalgebra of $\mathcal{M}.$ If $\mathcal{E}:\mathcal{M}\to\mathcal{N}$ is a conditional expectation, then there is a linear isometry $u:L_2(\mathcal{M})\to L_2(\mathcal{N}\overline{\otimes} B(\ell_2))$ such that
$$u(x)^{\ast}u(y)=\mathcal{E}(x^{\ast}y)\otimes e_{00}.$$
\end{prop}
\medskip

Now, we are fully equipped to present the proof of (DDD) and (DMT).

\begin{proof}[Proof of Theorem \ref{dual doob thm}] We apply Second Extrapolation Theorem to prove the distributional dual Doob inequality (DDD). Recall that $(\mathcal{M}_k)_{k\geq0}$ is an increasing sequence of von Neumann subalgebras of $\mathcal{M}$ such that the union $\bigcup_{k\geq 0} \mathcal{M}_k$ is weak$^{\ast}$ dense in $\mathcal{M}.$ Let $\mathcal{E}_{k}:\mathcal{M}\to\mathcal{M}_k$ be the conditional expectation and let $u_k:L_2(\mathcal{M})\to L_2(\mathcal{M}_k\overline{\otimes} B(\ell_2))\subset L_2(\mathcal{M}\overline{\otimes} B(\ell_2))$ be the mapping constructed in Proposition \ref{junge expectation theorem}. Then
\begin{equation*}\label{uk square}
|u_k(x)|^2=\mathcal{E}_k(x^{\ast}x)\otimes e_{00},\quad x\in L_2(\mathcal{M}).
\end{equation*}
In particular, we have
$$\|u_k(x)\|_p\leq\|x\|_p, \quad x\in(L_p\cap L_2)(\mathcal{M}),\quad 2\leq p<\infty.$$
Therefore, $u_k$ uniquely extends to a contractive linear mapping
$$u_k:L_p(\mathcal{M})\to L_p(\mathcal{M}\overline{\otimes} B(\ell_2)),\quad 2\leq p<\infty.$$

Consider the mapping $T$ which is formally given by the formula
\begin{equation}\label{doob t def}
T:x\to\sum_{k\geq0}u_k(x_{k0})\otimes e_{k0},\quad x\in L_p(\mathcal{M}\bar{\otimes}B(\ell_2)),
\end{equation}
where the sequence $(x_{kl})_{k,l\geq0}$ is constructed in Lemma \ref{lp tensor fact}. We claim that, for every  $x\in L_p(\mathcal{M}\bar{\otimes}B(\ell_2)),$ $2\leq p<\infty,$ the series on the right hand side in \eqref{doob t def} converges in $L_p(\mathcal{M}\bar{\otimes}B(\ell_2)\bar{\otimes}B(\ell_2)).$ Further, we claim that $$T:L_p(\mathcal{M}\bar{\otimes}B(\ell_2))\to L_p(\mathcal{M}\bar{\otimes}B(\ell_2)\bar{\otimes}B(\ell_2)),\quad 2\leq p<\infty,$$
is a well defined bounded mapping which satisfies Second Extrapolation Condition.

To see the claim, we Lemma \ref{lp tensor cor} and write
$$\Big\|\sum_{k\geq0}|x_{k0}|^2\Big\|_{\frac{p}{2}}^{\frac12}\leq\|x\|_p.$$
Now, using (DD), we write
$$\Big\|\sum_{k\geq0}\mathcal{E}_k(|x_{k0}|^2)\Big\|_{\frac{p}{2}}^{\frac12}\leq c_{{\rm abs}}p\Big\|\sum_{k\geq0}|x_{k0}|^2\Big\|_{\frac{p}{2}}^{\frac12}\leq c_{{\rm abs}}\|x\|_p.$$
Equivalently,
$$\Big\|\Big(\sum_{k\geq0}|u_k(x_{k0})|^2\Big)^{\frac12}\Big\|_p=\Big\|\Big(\sum_{k\geq0}\mathcal{E}_k(|x_{k0}|^2)\Big)^{\frac12}\Big\|_p\leq c_{{\rm abs}}\|x\|_p.$$
By Lemma \ref{convergence fact}, the series on the right hand side of \eqref{doob t def} converges in the space $L_p(\mathcal{M}\overline{\otimes}B(\ell_2)\bar{\otimes}B(\ell_2)),$ which means that the operator $T$ is well defined. Moreover, for all $2\leq p<\infty$,
$$\|Tx\|_p\leq c_{{\rm abs}}p\|x\|_p,\quad x\in L_p(\mathcal{M}\overline{\otimes}B(\ell_2)).$$
In other words, the operator $T$ satisfies Second Extrapolation Condition.

Applying Theorem \ref{second extrapolation theorem} to the operator $T$, we have
$$\mu^2(Tx)\prec\prec c_{{\rm abs}} (C^{\ast}\mu(x))^2, \quad x\in (\Lambda_{{\rm log}}\cap (L_2+L_{\infty}))(\mathcal{M}\overline{\otimes}B(\ell_2)).$$
Now we assume that
$$x=\sum_{k\geq0}a_k^{\frac12}\otimes e_{k0}\in (\Lambda_{{\rm log}}\cap (L_2+L_{\infty}))(\mathcal{M}\overline{\otimes}B(\ell_2)).$$
Applying the above inequality to $x,$ we obtain
$$\mu^2(\sum_{k\geq0}u_k(a_k^{\frac12})\otimes e_{k0})\prec\prec c_{{\rm abs}} (C^{\ast}\mu(\sum_{k\geq0}a_k^{\frac12}\otimes e_{k0}))^2$$
or, equivalently,
\begin{equation}\label{doob main eq}
\mu\Big(\Big|\sum_{k\geq0}u_k(a_k^{\frac12})\otimes e_{k0}\Big|^2\Big)\prec\prec c_{{\rm abs}} \Big(C^{\ast}\mu^{\frac12}\Big(\Big|\sum_{k\geq0}a_k^{\frac12}\otimes e_{k0}\Big|\Big)\Big)^2.
\end{equation}

It is immediate that
$$|\sum_{k\geq0}a_k^{\frac12}\otimes e_{k0}|^2=\sum_{k,l\geq0}a_k^{\frac12}a_l^{\frac12}\otimes e_{0k}e_{l0}=\sum_{k\geq0}a_k\otimes e_{00}=\big(\sum_{k\geq0}a_k\big)\otimes e_{00},$$
$$|\sum_{k\geq0}u_k(a_k^{\frac12})\otimes e_{k0}|^2=\sum_{k,l\geq0}u_k(a_k^{\frac12})^{\ast}u_l(a_l^{\frac12})\otimes e_{0k}e_{l0}=$$
$$=\sum_{k\geq0}u_k(a_k^{\frac12})^{\ast}u_k(a_k^{\frac12})\otimes e_{00}=\sum_{k\geq0}\mathcal{E}_k(a_k)\otimes e_{00}\otimes e_{00}=\big(\sum_{k\geq0}\mathcal{E}_ka_k\big)\otimes e_{00}\otimes e_{00}.$$
Substituting these expressions into \eqref{doob main eq}, we obtain
$$\mu\Big(\sum_{k\geq0}\mathcal{E}_ka_k\Big)\prec\prec c_{{\rm abs}} \Big(C^{\ast}\mu^{\frac12}\Big(\sum_{k\geq0}a_k\Big)\Big)^2.$$
This completes the proof.
\end{proof}

\begin{proof}[Proof of Theorem \ref{mt theorem}] Compared with (DST) and (DDD), the proof for distributional martingale transform estimates (DMT) is relatively easier. Consider the mapping $T:L_2(\mathcal{M})\to L_2(\mathcal{M})$ defined by the formula
$$T:x\to \sum_{k\geq0}\epsilon_k(\mathcal{E}_{k}x-\mathcal{E}_{{k-1}}x),\quad x\in L_2(\mathcal{M}).$$
By Theorem \ref{mt 11 theorem}, this mapping admits a bounded linear extension $T:L_1(\mathcal{M})\to L_{1,\infty}(\mathcal{M});$ in other words, we have
$$\|Tx\|_{1,\infty}\leq c_{{\rm abs}}\|x\|_1,\quad x\in L_1(\mathcal{M}).$$
Since $T$ on $L_2(\mathcal{M})$ is an isometry, it follows from Corollary \ref{mt 11 corollary} that there exists a bounded linear extension $T:L_p(\mathcal{M})\to L_p(\mathcal{M}),$ $1<p\leq 2,$ and that
$$\|Tx\|_p\leq c_{{\rm abs}}p'\|x\|_p,\quad x\in L_p(\mathcal{M}),\quad 1<p\leq 2.$$
Since $T=T^{\ast}$ on $L_2(\mathcal{M}),$  it follows by duality and a suitable analogue of Lemma \ref{first duality operator lemma} (or Lemma \ref{second duality operator lemma}) that there exists a bounded linear extension $T:L_p(\mathcal{M})\to L_p(\mathcal{M}),$ $2\leq p<\infty,$ and that
$$\|Tx\|_p\leq c_{{\rm abs}}p\|x\|_p,\quad x\in L_p(\mathcal{M}),\quad 2\leq p<\infty.$$
In other words, the operator $T$ satisfies First Extrapolation Condition. Therefore, Corollary \ref{first extrapolation corollary} applies to the operator $T$ and yields
$$\mu(Tx)\leq c_{{\rm abs}}S\mu(x),\quad x\in\Lambda_{{\rm log}}(\mathcal M)$$
which is the desired result.
\end{proof}

We now turn to the proof of distributional Burkholder-Gundy inequalities. The following operator form of Burkholder-Gundy inequality is a key ingredient in the proof of (Upper$-$DBG). The proof of this operator form of Burkholder-Gundy inequality depends on some estimates for the generalised martingale transforms whose proof are included in the Appendix (Theorem \ref{mt 11 theorem} and Corollary \ref{mt 11 corollary}).


\begin{prop}\label{bg best constant theorem} For every sequence $(x_k)_{k\geq0}\subset L_p(\mathcal{M}),$ $2\leq p<\infty,$ we have
$$\Big\|\sum_{k\geq0}\mathcal{E}_{k}x_k-\mathcal{E}_{{k-1}}x_k\Big\|_p\leq c_{{\rm abs}}p\Big\|\sum_{k\geq0}x_k\otimes\gamma_k\Big\|_p$$
whenever series on the right hand side is in $L_p(\mathcal{M}\bar{\otimes}\mathbb{F}_{\infty}).$ Here, $(\gamma_k)_{k\geq0}$ is the sequence of generators of the free group in the free group factor $L_{\infty}(\mathbb{F}_{\infty}).$
\end{prop}
\begin{proof} Suppose first that $(x_k)_{k\geq0}$ is a finite sequence, that is, $x_k=0$ for all sufficiently large $k.$ Recall the equality
$$\sum_{k\geq0}\tau(x_ky_k)=(\tau\otimes\tau_{\infty})\Big(\Big(\sum_{k\geq0}x_k\otimes\gamma_k\Big)\cdot \Big(\sum_{k\geq0}y_k\otimes\gamma_k^{-1}\Big)\Big).$$
It follows from H\"older inequality that
\begin{equation}\label{bgbc holder eq}
\big|\sum_{k\geq0}\tau(x_ky_k)\big|\leq\Big\|\sum_{k\geq0}x_k\otimes\gamma_k\Big\|_p\Big\|\sum_{k\geq0}y_k\otimes\gamma_k^{-1}\Big\|_{p'}.
\end{equation}
It follows from K\"othe duality of $L_p$ and $L_{p'}$ that
\begin{equation}\label{bgbc eq0}
\Big\|\sum_{k\geq0}\mathcal{E}_{k}x_k-\mathcal{E}_{k-1}x_k\Big\|_p=\sup_{\|y\|_{p'}\leq1}\Big|\tau\Big(\Big(\sum_{k\geq0}\mathcal{E}_{k}x_k-\mathcal{E}_{{k-1}}x_k\Big)y\Big)\Big|.
\end{equation}
Observe that
\begin{equation}\label{bgbc eq1}
\begin{split}
\tau\Big(\Big(\sum_{k\geq0}\mathcal{E}_{k}x_k-\mathcal{E}_{{k-1}}x_k\Big)y\Big)&
=\sum_{k\geq0}\tau\Big((\mathcal{E}_{k}x_k-\mathcal{E}_{{k-1}}x_k)y\Big)\\
&=\sum_{k\geq0}\tau(x_k\cdot\mathcal{E}_{k}y)-\tau(x_k\cdot\mathcal{E}_{{k-1}}y)\\
&=\sum_{k\geq0}\tau\Big(x_k\cdot (\mathcal{E}_{k}y-\mathcal{E}_{{k-1}}y)\Big).
\end{split}
\end{equation}
Combining \eqref{bgbc holder eq} and \eqref{bgbc eq1}, we obtain
$$\Big|\tau\Big(\Big(\sum_{k\geq0}\mathcal{E}_{k}x_k-\mathcal{E}_{{k-1}}x_k\Big)y\Big)\Big|\leq\Big\|\sum_{k\geq0}x_k\otimes\gamma_k\Big\|_p\cdot\Big\|\sum_{k\geq0}(\mathcal{E}_{k}y-\mathcal{E}_{{k-1}}y)\otimes\gamma_k^{-1}\Big\|_{p'}.$$
Substituting into \eqref{bgbc eq0}, we arrive at
$$\Big\|\sum_{k\geq0}\mathcal{E}_{k}x_k-\mathcal{E}_{{k-1}}x_k\Big\|_p\leq\Big\|\sum_{k\geq0}x_k\otimes\gamma_k\Big\|_p\cdot\sup_{\|y\|_{p'}\leq1}\Big\|\sum_{k\geq0}(\mathcal{E}_{k}y-\mathcal{E}_{{k-1}}y)\otimes\gamma_k^{-1}\Big\|_{p'}.$$
Applying Corollary \ref{mt 11 corollary}, we complete the proof for the case of finite sequence.

Now, let $(x_k)_{k\geq0}\subset L_p(\mathcal{M})$ be an infinite sequence. Let $L_{\infty}(\mathbb{F}_n)$ be the free group subfactor in $L_{\infty}(\mathbb{F}_{\infty})$ generated by $(\gamma_k)_{k=0}^n.$ Let $\mathbb{E}_n:L_2(\mathbb{F}_{\infty})\to L_2(\mathbb{F}_n)$ be the conditional expectation. By Proposition \ref{exp cpp prop}, the mapping $\mathbb{E}_n$ is completely positive and, therefore, completely contractive. Thus, ${\rm id}\otimes\mathbb{E}_n$ is a contraction both in $L_1$-norm and in $L_{\infty}$-norm. In particular, we have
$$\sum_{k=0}^nx_k\otimes\gamma_k=({\rm id}\otimes\mathbb{E}_n)(\sum_{k\geq0}x_k\otimes\gamma_k)\prec\prec\sum_{k\geq0}x_k\otimes\gamma_k.$$
Thus,
$$\Big\|\sum_{k=0}^nx_k\otimes\gamma_k\Big\|_p\leq \Big\|\sum_{k\geq0}x_k\otimes\gamma_k\Big\|_p,\quad n\geq0.$$
By the preceding paragraph,
$$\Big\|\sum_{k=0}^n\mathcal{E}_{k}x_k-\mathcal{E}_{{k-1}}x_k\Big\|_p\leq c_{{\rm abs}}p\Big\|\sum_{k\geq0}x_k\otimes\gamma_k\Big\|_p,\quad n\geq0.$$
Hence, the sequence
$$\Big(\sum_{k=0}^n\mathcal{E}_{k}x_k-\mathcal{E}_{{k-1}}x_k\Big)_{n\geq0}$$
is an $L_p$-bounded martingale. Since $1<p<\infty,$ it follows that the series
$$\sum_{k\geq0}\mathcal{E}_{k}x_k-\mathcal{E}_{{k-1}}x_k$$
converges in $L_p(\mathcal{M})$ and
$$\Big\|\sum_{k\geq 0}\mathcal{E}_{k}x_k-\mathcal{E}_{{k-1}}x_k\Big\|_p\leq c_{{\rm abs}}p\Big\|\sum_{k\geq0}x_k\otimes\gamma_k\Big\|_p,\quad n\geq0.$$
This completes the proof.
\end{proof}

We now are ready to prove the upper estimate in Theorem \ref{bg theorem}; namely, the inequality (Upper$-$DBG).

\begin{proof}[Proof of the upper estimate in Theorem \ref{bg theorem}] Let $2\leq p<\infty$. By Proposition \ref{bg best constant theorem}, we have
\begin{equation}\label{bgbc eq}
\Big\|\sum_{k\geq0}\mathcal{E}_{k}x_k-\mathcal{E}_{{k-1}}x_k\Big\|_p\leq c_{{\rm abs}}p\Big\|\sum_{k\geq0}x_k\otimes\gamma_k\Big\|_p
\end{equation}
whenever the right hand side is finite. Set
$$A:\sum_{k\geq0}x_k\otimes\gamma_k\to \sum_{k\geq0}(\mathcal{E}_{k}x_k-\mathcal{E}_{{k-1}}x_k),$$
and let $P:L_2(\mathbb{F}_{\infty})\to L_2(\mathbb{F}_{\infty})$ be the projection onto the linear span of the sequence $(\gamma_k)_{k\geq0}.$ One can rewrite \eqref{bgbc eq} as
\begin{equation}\label{bg eq1}
\|Az\|_p\leq c_{{\rm abs}}p\|z\|_p,\quad z\in L_p(\mathcal{M}\bar{\otimes}\mathbb{F}_{\infty}),\quad ({\rm id}\otimes P)z=z.
\end{equation}
By Proposition 1.1 in \cite{HP} and complex interpolation, we have
\begin{equation}\label{bg eq2}
\|{\rm id}\otimes P\|_{L_p(\mathcal{M}\bar{\otimes}\mathbb{F}_{\infty})\to L_p(\mathcal{M}\bar{\otimes}\mathbb{F}_{\infty})}\leq 2,\quad 2\leq p\leq\infty.
\end{equation}

Define the operator $T$ by setting
$$T=A\circ ({\rm id}\otimes P).$$
Combining \eqref{bg eq1} and \eqref{bg eq2}, we obtain
$$\|Tz\|_p\leq c_{{\rm abs}}p\|z\|_p,\quad z\in L_p(\mathcal{M}\bar{\otimes}\mathbb{F}_{\infty}),\quad 2\leq p<\infty.$$
It then follows that
$$\|Tz\|_{L_p(\mathcal{M},2\tau)}=2^{\frac1p}\|Tz\|_p\leq c_{{\rm abs}}p\|z\|_p,\quad z\in L_p(\mathcal{M}\bar{\otimes}\mathbb{F}_{\infty}),\quad 2\leq p<\infty.$$
We conclude that the operator $T:L_p(\mathcal{M}\bar{\otimes}\mathbb{F}_{\infty})\to L_p(\mathcal{M},2\tau)$ satisfies Second Extrapolation Condition so that Theorem \ref{second extrapolation theorem} is applicable. By Theorem \ref{second extrapolation theorem}, we have
$$\sigma_2\mu^2(Tz)\prec\prec c_{{\rm abs}}(C^{\ast}\mu(z))^2,\quad z\in (\Lambda_{{\rm log}}\cap (L_2+L_{\infty}))(\mathcal{M}\bar{\otimes}\mathbb{F}_{\infty}).$$
Now, take $x\in \Lambda_{{\rm log}}(\mathcal{M})$ and let $(x_k)_{k\geq0}$ be the sequence of respective martingale differences such that
$$\sum_{k\geq0}x_k^{\ast}x_k+x_kx_k^{\ast}\in (L_1+L_{\infty})(\mathcal{M}),\quad \big(\sum_{k\geq0}x_k^{\ast}x_k+x_kx_k^{\ast}\big)^{\frac12}\in\Lambda_{{\rm log}}(\mathcal{M}).$$
Set
$$z=\sum_{k\geq0}x_k\otimes\gamma_k,$$
where the series converges in measure and note that
$$({\rm id}\otimes P)z=z,\quad Tz=x.$$
Therefore,
$$\sigma_2\mu^2(x)\prec\prec c_{{\rm abs}}\Big(C^{\ast}\mu\Big(\sum_{k\geq0}x_k\otimes\gamma_k\Big)\Big)^2.$$
Lemma 2.1 in \cite{Dirksen-Ricard} asserts that
$$\mu\Big(\sum_{k\geq0}x_k\otimes\gamma_k\Big)\leq c_{{\rm abs}}\sigma_2\mu^{\frac12}\Big(\sum_{k\geq0}x_kx_k^{\ast}+x_k^{\ast}x_k\Big).$$
Combining these estimates, we arrive at
$$\sigma_2\mu^2(x)\prec\prec c_{{\rm abs}}\Big(C^{\ast}\sigma_2\mu^{\frac12}\Big(\sum_{k\geq0}x_kx_k^{\ast}+x_k^{\ast}x_k\Big)\Big)^2.$$
Since $C^{\ast}$ commutes with $\sigma_2,$ the upper estimate in Theorem \ref{bg theorem} follows.
\end{proof}

The proof of the lower estimate in Theorem \ref{bg theorem} (i.e., (Lower$-$DBG)) is given below. It is easier than the upper one since in this case (BG) itself suggests an operator which is an obvious candidate for the application of Second Extrapolation Theorem.

\begin{proof}[Proof of the lower estimate in Theorem \ref{bg theorem}] Define the operator $T$ by the formula
$$Tx=\sum_{k\geq0}(\mathcal{E}_{k}x-\mathcal{E}_{{k-1}}x)\otimes e_{k,0}.$$
By (BG), we have (recall the notation $x_k=\mathcal{E}_{k}x-\mathcal{E}_{{k-1}}x,$ $k\geq0,$ in (BG))
$$\|Tx\|_p=\Big\|(\sum_{k\geq0}|x_k|^2)^{\frac12}\Big\|_p\leq c_{{\rm abs}}p\|x\|_p,\quad x\in L_p(\mathcal{M}),\quad 2\leq p<\infty.$$
We conclude that the operator $$T:L_p(\mathcal{M})\to L_p(\mathcal{M}\bar{\otimes}B(\ell_2))$$ satisfies the Second Extrapolation Condition. By Theorem \ref{second extrapolation theorem}, we have
$$\mu^2(Tx)\prec\prec c_{{\rm abs}}(C^{\ast}\mu(x))^2,\quad x\in (\Lambda_{{\rm log}}\cap (L_2+L_{\infty}))(\mathcal{M}).$$
Hence,
$$\mu\Big(\sum_{k\geq0}x_k^{\ast}x_k\Big)\prec\prec c_{{\rm abs}}(C^{\ast}\mu(x))^2,\quad x\in (\Lambda_{{\rm log}}\cap (L_2+L_{\infty}))(\mathcal{M}).$$
Similarly, we have
$$\mu\Big(\sum_{k\geq0}x_kx_k^{\ast}\Big)\prec\prec c_{{\rm abs}}(C^{\ast}\mu(x))^2,\quad x\in (\Lambda_{{\rm log}}\cap (L_2+L_{\infty}))(\mathcal{M}).$$
Combining the estimates above, we complete the proof.
\end{proof}

\section{Distributional estimates are optimal for $B(L_2(0,\infty))$}\label{optimal section}

In this section, we prove that the distributional Stein inequality is optimal in the case $\mathcal{M}= B(L_2(0,\infty)).$

Let us first recall how (DST) can be written for this particular $\mathcal{M}.$ It is important that $\mu(x),$ $x\in\mathcal{M},$ is constant on every interval $(n,n+1).$ Therefore,
\begin{align*}
\mu\Big(n,\Big(\sum_{k\geq0}|\mathcal{E}_kx_k|^2\Big)^{\frac12}\Big)
&=\mu\Big(n+1-0,\Big(\sum_{k\geq0}|\mathcal{E}_kx_k|^2\Big)^{\frac12}\Big)\\
&\leq c_{{\rm abs}}\Big(S\mu\Big(\Big(\sum_{k\geq0}|x_k|^2\Big)^{\frac12}\Big)\Big)(n+1-0)\\
& \approx c_{{\rm abs}}\Big(S_d\mu\Big(\Big(\sum_{k\geq0}|x_k|^2\Big)^{\frac12}\Big)\Big)(n),
\end{align*}
where the notion $n+1-0$ stands for the number infinitely close to $n+1$ from the left. In other words, we can replace $S$ with $S_d$ in the statement of (DST).

\begin{thm}\label{stein optimal thm} Let $\mathcal{M}=B(L_2(0,\infty))$. For every $x\in\Lambda_{{\rm log}}(\mathbb{Z}_+),$ there exists a net of filtrations $(\mathcal{F}_t)_{t>0}=((\mathcal{M}_k^t)_{k\geq0})_{t>0}$ where $\mathcal M_k^t\subset \mathcal{M},$
a net of sequences  $((a_k^t)_{k\geq0})_{t>0},$
and elements $z_1\in\Lambda_{{\rm log}}(\mathcal{M})$ and $z_2\in L_{\infty}(\mathcal{M})$ such that
\begin{enumerate}[{\rm (i)}]
\item $\mu(z_1)=\mu(x)$ and
$$\Big(\sum_{k\geq0}a_k^t(a_k^t)^{\ast}\Big)^{\frac12}=z_1,\quad t>0;$$
\item $\mu(z_2)\geq c_{{\rm abs}}S_d\mu(x)$ and
$$\lim_{t\to0}\Big(\sum_{k\geq0}\mathcal{E}_{\mathcal{M}_k^t}(a_k^t)\mathcal{E}_{\mathcal{M}_k^t}(a_k^t)^{\ast}\Big)^{\frac12}=z_2.$$
\end{enumerate}
\end{thm}

Let us first provide the construction of the net of filtrations. For every $t>0,$ we define the filtration $\mathcal{F}_t=(\mathcal{M}_k^t)_{k\geq0}$ in $\mathcal{M}$ as follows.
\begin{enumerate}[{\rm (i)}]
\item Let $p_s=M_{\chi_{(0,s)}}$ and let $\mathcal{D}^t$ be the von Neumann subalgebra generated by $(p_{kt})_{k\in\mathbb{Z}}.$
\item For $k\geq 0$, set $\mathcal{M}_k^t=p_{kt}\mathcal{M}p_{kt}+(\mathcal{D}^t)'.$
\end{enumerate}
For every $t>0$, we also introduce the following operator:
$$P_tx=\sum_{\substack{k,l\in\mathbb{Z}_+\\ k\leq l}}\Big(p_{(k+1)t}-p_{kt})x(p_{(l+1)t}-p_{lt}\Big),\quad x\in L_2(\mathcal{M}).$$

\begin{lem}\label{pt convergence lemma} $P_ta\to Pa$ in the uniform norm as $t\to0$ for every $a\in\Lambda_{{\rm log}}(\mathcal{M}).$ Here, $P$ is the triangular truncation operator defined in Subsection \ref{truncation subsection}.
\end{lem}
\begin{proof} First, note that the net of operators $(P_t)_{t>0}:L_2\to L_2$ extends to a uniformly bounded net $(P_t)_{t>0}:L_1\to L_{1,\infty}$ (see Theorem 1.4 in \cite{DDPS-IEOT}). Next, note that $P_t(p_{kt}xp_{kt})\in L_1$ and that
$$\|P_tx-P_t(p_{kt}xp_{kt})\|_{1,\infty}\leq\|P_t\|_{L_1\to L_{1,\infty}}\|x-p_{kt}xp_{kt}\|_1\leq c_{{\rm abs}}\|x-p_{kt}xp_{kt}\|_1\to 0$$
as $k\to\infty.$ Hence, the net $(P_t)_{t>0}:L_1\to (L_{1,\infty})_0$ is uniformly bounded (here, $(L_{1,\infty})_0$ is the closure of $L_1$ in $L_{1,\infty}$). Since $\mathcal{M}=B(L_2(0,\infty)),$ it follows that $L_{1,\infty}\subset M_{{\rm log}}.$ Consequently,  the net $(P_t)_{t>0}:L_1\to (M_{{\rm log}})_0$ is uniformly bounded (here, $(M_{{\rm log}})_0$ is the closure of $L_1$ in $M_{{\rm log}}$). Taking the adjoint operators, we obtain that the net $(P_t^{\ast})_{t>0}:((M_{{\rm log}})_0)^{\ast}\to L_1^{\ast}$ is uniformly bounded. Since the spaces $(M_{{\rm log}})_0$ and $L_1$ are separable, it follows that their Banach duals coincide with their K\"othe duals. Thus, the net $(P_t^{\ast})_{t>0}:\Lambda_{{\rm log}}\to L_{\infty}$ is uniformly bounded.

Since $L_2$ is dense in $\Lambda_{{\rm log}},$ it follows that $P_t^{\ast}:\Lambda_{{\rm log}}\to L_{\infty}$ is the unique extension of $P_t^{\ast}:L_2\to L_2.$ However, on $L_2(\mathcal{M}),$ we have
$$P_t^{\ast}+P_t={\rm id}+D_t,$$
$$D_tx=\sum_{k\geq0}\Big(p_{(k+1)t}-p_{kt})x(p_{(k+1)t}-p_{t}\Big),\quad x\in L_2(\mathcal{M}).$$
The net $(D_t)_{t>0}:\Lambda_{{\rm log}}\to L_{\infty}$ is, obviously, uniformly bounded. Hence, the net $(P_t)_{t>0}:\Lambda_{{\rm log}}\to L_{\infty}$ is uniformly bounded.

Take $a\in\Lambda_{{\rm log}}(\mathcal{M})$ and fix $\varepsilon>0.$ Since the norm in $\Lambda_{{\rm log}}(\mathcal{M})$ is order-continuous, it follows that $L_2(\mathcal{M})$ is dense in $\Lambda_{{\rm log}}(\mathcal{M}).$ Choose $b\in L_2(\mathcal{M})$ such that $\|a-b\|_{\Lambda_{\rm log}}<\varepsilon.$ Obviously, $P_tb\to Pb$ in $L_2.$ Fix $t(\varepsilon)$ such that $\|P_tb-Pb\|_2<\varepsilon$ for $t<t(\varepsilon).$ We have
\begin{align*}
\|P_ta-Pa\|_{\infty}&\leq\|P_t(a-b)\|_{\infty}+\|P_tb-Pb\|_{\infty}+\|Pa-Pb\|_{\infty}\\
&\leq \|P_t\|_{\Lambda_{\rm log}\to L_{\infty}}\|a-b\|_{\Lambda_{\rm log}}+\|P_tb-Pb\|_2+\|P\|_{\Lambda_{{\rm log}}\to L_{\infty}}\|a-b\|_{\Lambda_{\rm log}}.
\end{align*}
Hence,
$$\|P_ta-Pa\|_{\infty}\leq c_{{\rm abs}}\varepsilon,\quad 0<t<t(\varepsilon).$$
Since $\varepsilon>0$ is arbitrarily small, the assertion follows.
\end{proof}

\begin{lem}\label{jx stein lemma} Let $a\in\Lambda_{{\rm log}}(\mathcal{M}).$ For every $t>0,$ there exists a sequence $(a_k^t)_{k\geq0}$ such that
$$\sum_{k\geq0}a_k^t(a_k^t)^{\ast}=aa^{\ast},\quad \sum_{k\geq0}\mathcal{E}_{\mathcal{M}_k^t}(a_k^t)\mathcal{E}_{\mathcal{M}_k^t}(a_k^t)^{\ast}=P_t(a)P_t(a)^{\ast}.$$
\end{lem}
\begin{proof} Set $a_k^t=a(p_{(k+1)t}-p_{kt}).$ Obviously,
$$a_k^t(a_k^t)^{\ast}=a\big(p_{(k+1)t}-p_{kt}\big)a^{\ast}.$$
Thus,
$$\sum_{k\geq0}a_k^t(a_k^t)^{\ast}=\sum_{k\geq0}a\big(p_{(k+1)t}-p_{kt}\big)a^{\ast}=a\Big(\sum_{k\geq0}(p_{(k+1)t}-p_{kt})\Big)a^{\ast}=aa^{\ast}.$$
On the other hand, we have
$$\mathcal{E}_{\mathcal{M}_k^t}x=p_{(k+1)t}xp_{(k+1)t}+\sum_{l>k}(p_{(l+1)t}-p_{lt})x(p_{(l+1)t}-p_{lt}),\quad x\in\mathcal{M}.$$
Thus,
$$\mathcal{E}_{\mathcal{M}_k^t}(a_k^t)=p_{(k+1)t}a(p_{(k+1)t}-p_{kt})=P_t(a)(p_{(k+1)t}-p_{kt}).$$
Again, we have
$$\mathcal{E}_{\mathcal{M}_k^t}(a_k^t)\mathcal{E}_{\mathcal{M}_k^t}(a_k^t)^{\ast}=P_t(a)(p_{(k+1)t}-p_{kt})P_t(a)^{\ast}.$$
Therefore,
$$\sum_{k\geq0}\mathcal{E}_{\mathcal{M}_k^t}(a_k^t)\mathcal{E}_{\mathcal{M}_k^t}(a_k^t)^{\ast}=\sum_{k\geq0}P_t(a)(p_{(k+1)t}-p_{kt})P_t(a)^{\ast}=P_t(a)P_t(a)^{\ast}.$$
The proof is complete.
\end{proof}

Now we are in a position to prove Theorem \ref{stein optimal thm}.
\begin{proof}[Proof of Theorem \ref{stein optimal thm}] Fix $x\in\Lambda_{{\rm log}}(\mathbb{Z}_+)$ and let $a$ be as in Corollary 5 in \cite{STZ-corr}. Set $z_1=|a^{\ast}|$ and $z_2=|(Pa)^{\ast}|.$ Recall that $a\in\Lambda_{{\rm log}}(\mathcal{M})$ so that $Pa\in L_{\infty}(\mathcal{M}).$
	
For every $t>0,$ let $(a_k^t)_{k\geq0}$ be a sequence constructed in Lemma \ref{jx stein lemma}. By Lemma \ref{jx stein lemma} and Lemma \ref{pt convergence lemma}, we have (the limit is taken in the uniform norm)
$$\Big(\sum_{k\geq0}\mathcal{E}_{\mathcal{M}_k^t}(a_k^t)\mathcal{E}_{\mathcal{M}_k^t}(a_k^t)^{\ast}\Big)^{\frac12}=|P_t(a)^{\ast}|\to |P(a)^{\ast}|=z_2,\quad t\downarrow 0.$$
By Lemma \ref{jx stein lemma}, we have
$$\Big(\sum_{k\geq0}(a_k^t)(a_k^t)^{\ast}\Big)^{\frac12}=|a^{\ast}|=z_1,\quad t>0.$$
By Corollary 5 in \cite{STZ-corr}, we have $\mu(z_1)=\mu(x)$ and $\mu(z_2)\geq c_{{\rm abs}}S_d\mu(x).$
\end{proof}

\section{Martingale inequalities in quasi-Banach spaces}\label{application section}

In this section, we apply Theorems \ref{stein thm}, \ref{dual doob thm}, \ref{mt theorem} and \ref{bg theorem} to obtain new martingale inequalities in symmetric quasi-Banach operator spaces (in particular, in Orlicz spaces, including the endpoint cases).

\subsection{Estimates in general quasi-Banach spaces}\label{application-1}

The following assertions are straightforward corollaries of Theorems \ref{stein thm}, \ref{dual doob thm}, \ref{mt theorem} and \ref{bg theorem}, respectively.

\begin{thm}[Stein inequality in quasi-Banach spaces]\label{qst} Let $E$ and $F$ be symmetric quasi-Banach function spaces on $(0,\infty).$ If $S:E\to F,$ then, for every sequence $(x_k)_{k\geq0}\subset E(\mathcal{M}),$ we have
$$\big\|(\sum_{k\geq0}|\mathcal{E}_{k}x_k|^2)^{\frac12}\big\|_F\leq c_{{\rm abs}}\|S\|_{E\to F}\big\|(\sum_{k\geq0}|x_k|^2)^{\frac12}\big\|_E.$$
\end{thm}
\begin{proof} By (DST), we have
$$\big\|(\sum_{k\geq0}|\mathcal{E}_{k}x_k|^2)^{\frac12}\big\|_F\leq c_{{\rm abs}}\Big\|S\mu^{\frac12}\Big(\sum_{k\geq0}|x_k|^2\Big)\Big\|_F\leq c_{{\rm abs}}\|S\|_{E\to F}\big\|(\sum_{k\geq0}|x_k|^2)^{\frac12}\big\|_E.$$
\end{proof}

\begin{thm}[Dual Doob inequality in quasi-Banach spaces]\label{qdd} Let $E$ and $F$ be symmetric quasi-Banach function spaces on $(0,\infty).$ If $F\in{\rm Int}(L_1,L_{\infty})$ and if $C^{\ast}:E^{(2)}\to F^{(2)},$ then for every positive sequence $(a_k)_{k\geq0}\subset E(\mathcal{M}),$ we have
$$\|\sum_{k\geq0}\mathcal{E}_ka_k\|_F\leq c_{{\rm abs}}\|C^{\ast}\|_{E^{(2)}\to F^{(2)}}^2\|\sum_{k\geq0}a_k\|_E.$$
\end{thm}
\begin{proof} Since $F\in{\rm Int}(L_1,L_{\infty}),$ it follows that the norm in $F$ is monotone with respect to the Hardy-Littlewood-Polya submajorization. By (DDD), we have
\begin{align*}
\|\sum_{k\geq0}\mathcal{E}_ka_k\|_F&\leq c_{{\rm abs}}\Big\|\Big(C^{\ast}\mu^{\frac12}\Big(\sum_{k\geq 0}a_k\Big)\Big)^2\Big\|_F=c_{{\rm abs}}\Big\|C^{\ast}\mu^{\frac12}\Big(\sum_{k\geq 0}a_k\Big)\Big\|_{F^{(2)}}^2\\
&\leq c_{{\rm abs}}\|C^{\ast}\|_{E^{(2)}\to F^{(2)}}^2\Big\|\mu^{\frac12}\Big(\sum_{k\geq 0}a_k\Big)\Big\|_{E^{(2)}}^2\\
&=c_{{\rm abs}}\|C^{\ast}\|_{E^{(2)}\to F^{(2)}}^2\|\sum_{k\geq0}a_k\|_E.
\end{align*}
\end{proof}

\begin{thm}[Martingale Transform estimate in quasi-Banach spaces]\label{qmt} Let $E$ and $F$ be symmetric quasi-Banach function spaces on $(0,\infty).$ If $S:E\to F,$ then, for every sequence $x\in E(\mathcal{M}),$ we have
$$\|\sum_{k\geq0}\epsilon_k\big(\mathcal{E}_{k}x-\mathcal{E}_{{k-1}}x\big)\|_F\leq c_{{\rm abs}}\|S\|_{E\to F}\|x\|_E.$$
\end{thm}
\begin{proof} By (DMT), we have
$$\|\sum_{k\geq0}\epsilon_k\big(\mathcal{E}_{k}x-\mathcal{E}_{{k-1}}x\big)\|_F\leq c_{{\rm abs}}\|S\mu(x)\|_F\leq c_{{\rm abs}}\|S\|_{E\to F}\|x\|_E.$$
\end{proof}

The next theorem is a vast generalisation of Theorem 1.3 in \cite{JSZZ}.

\begin{thm}[Burkholder-Gundy inequality in quasi-Banach spaces]\label{qbg} Let $E$ and $F$ be symmetric quasi-Banach function spaces on $(0,\infty).$ Suppose that $F\in{\rm Int}(L_2,L_{\infty})$ and $C^{\ast}:E\to F.$ Let $x\in (L_2+L_{\infty})(\mathcal{M})$ and let $(x_k)_{k\geq0}$ be the respective sequence of martingale differences (i.e., $x_0=\mathcal{E}_{_0} x$ and $x_k=\mathcal{E}_{k} x-\mathcal{E}_{k-1}x$ for $k\geq 1$).
\begin{enumerate}[{\rm (i)}]
\item If $(x_k)_{k\geq0}\subset E(\mathcal{M}),$ then $x\in F(\mathcal{M})$ and
$$\|x\|_F\leq c_{{\rm abs}}\|C^{\ast}\|_{E\to F}\Big(\big\|\big(\sum_{k\geq0}|x_k|^2\big)^{\frac12}\big\|_E+\big\|\big(\sum_{k\geq0}|x_k^{\ast}|^2\big)^{\frac12}\big\|_E\Big).$$
\item If $x\in E(\mathcal{M}),$ then
$$\big\|\big(\sum_{k\geq0}|x_k|^2\big)^{\frac12}\big\|_F+\big\|\big(\sum_{k\geq0}|x_k^{\ast}|^2\big)^{\frac12}\big\|_F\leq c_{{\rm abs}}\|C^{\ast}\|_{E\to F}\|x\|_E.$$
\end{enumerate}
\end{thm}
\begin{proof} By the assumption $F\in{\rm Int}(L_2,L_{\infty})$ and by the Lorentz-Shimogaki theorem, if $a\in F$ and if $a\in L_0$ are such that $\mu^2(b)\prec\prec\mu^2(a),$ then $b\in F$ and $\|b\|_F\leq\|a\|_F.$

By the upper inequality (Upper$-$DBG), we have
\begin{align*}
\|x\|_F\leq &c_{{\rm abs}}\Big\|C^{\ast}\mu^{\frac12}\Big(\sum\limits_{k\geq 0}x_kx_k^{\ast}+x_k^{\ast}x_k\Big)\Big\|_F\\
&\leq c_{{\rm abs}}\|C^{\ast}\|_{E\to F}\Big\|\mu^{\frac12}\Big(\sum\limits_{k\geq 0}x_kx_k^{\ast}+x_k^{\ast}x_k\Big)\Big\|_F\\
&=c_{{\rm abs}}\|C^{\ast}\|_{E\to F}\Big\|\Big(\sum\limits_{k\geq 0}x_kx_k^{\ast}+x_k^{\ast}x_k\Big)^{\frac12}\Big\|_F.
\end{align*}
On the other hand, the lower inequality in (Lower$-$DBG) gives that
$$\Big\|\Big(\sum\limits_{k\geq 0}x_kx_k^{\ast}+x_k^{\ast}x_k\Big)^{\frac12}\Big\|_F\leq  c_{{\rm abs}}\big\|C^{\ast}\mu(x)\big\|_F\leq c_{{\rm abs}}\|C^{\ast}\|_{E\to F}\|x\|_E.$$
Obviously,
$$\Big\|\Big(\sum\limits_{k\geq 0}x_kx_k^{\ast}+x_k^{\ast}x_k\Big)^{\frac12}\Big\|_F\approx \big\|\big(\sum_{k\geq0}|x_k|^2\big)^{\frac12}\big\|_E+\big\|\big(\sum_{k\geq0}|x_k^{\ast}|^2\big)^{\frac12}\big\|_E.$$
The proof is complete by combining the above estimates.
\end{proof}

\begin{rem} Setting $E=F=L_p(0,\infty)$ and using Lemma \ref{cesaro-norm}, we obtain that
\begin{enumerate}[{\rm (i)}]
\item Theorem \ref{qst} reduces to (ST);
\item Theorem \ref{qdd} reduced to (DD);
\item Theorem \ref{qmt} reduces to (MT);
\item Theorem \ref{qbg} provides optimal constants in (BG).
\end{enumerate}
\end{rem}

\begin{rem} Let $L\log L$ be the Orlicz space $L_{\Phi}$ with $\Phi(t)=t\log(1+t), t>0$. Set $E=L\log L(0,1)$ and $F=L_1(0,1)$, and note that $S$ is bounded from $L\log L(0,1)$ to $L_1(0,1).$ Thus, by Theorem \ref{qst}, we obtain that
$$\big\|(\sum_{k\geq0}|\mathcal{E}_{k}x_k|^2)^{\frac12}\big\|_{L_1}\leq c_{{\rm abs}}\big\|(\sum_{k\geq0}|x_k|^2)^{\frac12}\big\|_{L\log L}.$$
\end{rem}

\subsection{Example: Orlicz and weak Orlicz spaces}
In this subsection,  we establish some new noncommutative martingale inequalities which have some main interest in view of the endpoint case.
We begin with two lemmas. The special case $\mathcal M=B(H)$ of the following lemma was proved in \cite[Proposition 3.14]{DFWW}. The method applied there is much more complicated than ours.

\begin{lem}\label{dfww lemma} If $z=(\frac1k)_{k\geq1},$ then
$$\frac12C\mu(x)\leq\mu(x\otimes z)\leq C\mu(x),\quad x\in (L_1+L_\infty)(\mathcal M).$$
\end{lem}
\begin{proof} Let $Z(t)=t^{-1},$ $t>0.$ The following crucial fact will be used several times:
$$\mu(x\otimes Z)=\|x\|_1Z.$$

We first prove the right hand side inequality. Fix $t>0$ and set
$$x_1=(|x|-\mu(t,x))_+\quad \mbox{and}\quad x_2=\min\{|x|,\mu(t,x)\}.$$
Obviously, $x_1+x_2=|x|.$ Therefore,
$$\mu(t,x\otimes z)\leq\mu(t,x_1\otimes z)+\mu(0,x_2\otimes z).$$
Clearly, $\mu(0,x_2\otimes z)=\mu(t,x)$ and
\begin{align*}
\mu(t,x_1\otimes z)&\leq\mu(t,x_1\otimes Z)=\frac1t\|x_1\|_1\\
&=\frac1{t}\int_0^t(\mu(s,x)-\mu(t,x))ds.
\end{align*}
Consequently,
$$\mu(t,x\otimes z)\leq\frac1{t}\int_0^t(\mu(s,x)-\mu(t,x))ds+\mu(t,x)=(C\mu(x))(t).$$
Since $t>0$ is arbitrary, the right hand side inequality follows.

Now we turn to the left hand side inequality. For every $t>0,$ we have
\begin{align*}
\frac12\big(C\mu(x)\big)(t)=\frac1{2t}\|\mu(x)\chi_{(0,t)}\|_1=\mu\left(2t,\mu(x)\chi_{(0,t)}\otimes Z\right).
\end{align*}
Then it follows from \eqref{singular-triangle} that
$$\frac12(C\mu(x))(t)\leq \mu\left(t,\mu(x)\chi_{(0,t)}\otimes Z\chi_{(0,1)}\right)+\mu\left(t,\mu(x)\chi_{(0,t)}\otimes Z\chi_{(1,\infty)}\right).$$
Observe that
$\mu(x)\chi_{(0,t)}\otimes Z\chi_{(0,1)}$
is supported on a set of measure $t.$ Therefore,
$$\mu\left(t,\mu(x)\chi_{(0,t)}\otimes Z\chi_{(0,1)}\right)=0.$$
Meanwhile, we have
\begin{align*}
\mu\left(t,\mu(x)\chi_{(0,t)}\otimes Z\chi_{(1,\infty)}\right)&\leq\mu\left(t,x\otimes Z\chi_{(1,\infty)}\right)\leq\mu\left(t,x\otimes z\right).
\end{align*}
Combining the last $3$ inequalities, we obtain
$$\frac12(C\mu(x))(t)\leq\mu(t,x\otimes z).$$
Since $t>0$ is arbitrary, the left hand side inequality follows.
\end{proof}

\begin{lem}\label{phi-lem}
Let $\Phi$ be an Orlicz function.
\begin{enumerate}[{\rm (i)}]
\item $C:L_{\Phi}\to L_{\Phi,\infty}$ is bounded.
\item If $\Phi$ is $q$-concave for some $1\leq q<\infty$, then $C^{\ast}:L_{\Phi}\to L_{\Phi}$ is bounded.
\item If $\Phi$ is $q$-concave for some $1\leq q<\infty$, then $S:L_{\Phi}\to L_{\Phi,\infty}$ is bounded.
\end{enumerate}
\end{lem}
\begin{proof} Let $x$ be such that $\|x\|_{\Phi}\leq 1.$ It follows that
$$\sum_{k\geq1}\Phi(k\lambda)\cdot\int_{k\lambda<|x|\leq(k+1)\lambda}1\leq\sum_{k\geq1}\int_{k\lambda<|x|\leq(k+1)\lambda}\Phi(|x|)\leq 1.$$
Since $\Phi(k\lambda)\geq k\Phi(\lambda)$ for all $k\geq1,$ it follows that
$$\sum_{k\geq1}k\cdot (n_{|x|}(k\lambda)-n_{|x|}((k+1)\lambda))=\sum_{k\geq1}k\cdot\int_{k\lambda<|x|\leq(k+1)\lambda}1\leq\frac1{\Phi(\lambda)}.$$
Using summation by parts, we obtain
$$\sum_{k\geq1}n_{|x|}(k\lambda)\leq\frac1{\Phi(\lambda)}.$$
If $z=(\frac1k)_{k\geq1},$ then
$$n_{|x|\otimes z}(\lambda)=\sum_{k\geq1}n_{|x|}(k\lambda)\leq\frac1{\Phi(\lambda)}.$$
In other words,
$$\|x\otimes z\|_{\Phi,\infty}\leq 1.$$
Let now $x\in L_{\Phi}$ be arbitrary. By homogeneity, we have
$$\|x\otimes z\|_{\Phi,\infty}\leq\|x\|_{\Phi}.$$
By Lemma \ref{dfww lemma}, we have $\mu(Cx)\leq 2\mu(x\otimes z).$ Therefore,
$$\|Cx\|_{\Phi,\infty}\leq 2\|x\|_{\Phi}.$$
This proves the first assertion.

If $\Phi$ is $q$-concave for some $1\leq q<\infty$, then $L_{\Phi}$ is an interpolation space between $L_1$ and $L_{2q}.$ Since $C^{\ast}$ is bounded on both $L_1$ and $L_{2q},$ it follows that $C^{\ast}$ is bounded on $L_{\Phi}.$ This proves the second assertion.

The last assertion is a combination of the first two.
\end{proof}

Using Lemma \ref{phi-lem}, we immediately get the following corollaries from the results in Subsection \ref{application-1}.

\begin{cor} Let $1\leq q<\infty$ and $\Phi$ be a $q$-concave Orlicz function. For every sequence $(x_k)_{k\geq 0}\subset L_{\Phi}(\mathcal{M}),$ we have
$$\Big\|\Big(\sum_{k\geq 0}|\mathcal E_{k}(x_k)|^2\Big)^{\frac12}\Big\|_{\Phi,\infty}\lesssim_{\Phi}\Big\|\Big(\sum_{k\geq 0}|x_k|^2\Big)^{\frac12}\Big\|_{\Phi}.$$
\end{cor}

\begin{cor} Let $1\leq q<\infty$ and $\Phi$ be a $q$-concave Orlicz function. For every positive sequence $(a_k)_{k\geq 0}\subset L_{\Phi}(\mathcal{M})$ we have
$$\|\sum_{k\geq 0}\mathcal{E}_ka_k\|_{\Phi}\lesssim_{\Phi}\|\sum_{k\geq0}a_k\|_{\Phi}.$$
\end{cor}

\begin{cor}Let $1\leq q<\infty$ and $\Phi$ be a $q$-concave Orlicz function.  For every $x\in L_{\Phi}(\mathcal{M})$ and for every choice of signs $(\epsilon_k)_{k\geq0}$ we have
$$\Big\|\sum_{k\geq0}\epsilon_k\big(\mathcal{E}_kx-\mathcal{E}_{k-1}x\big)\Big\|_{\Phi,\infty}\lesssim_{\Phi} \|x\|_{\Phi}.$$
\end{cor}

\begin{cor} Let $2\leq q<\infty$ and $\Phi$ be a $2$-convex and $q$-concave Orlicz function. For every $x\in L_{\Phi}(\mathcal{M})$ we have (here, $x_0=\mathcal{E}_0 x$ and $x_k=\mathcal{E}_kx-\mathcal{E}_{k-1}x$ for $k\geq 1$)
$$\|x\|_{\Phi}\approx_{\Phi}\big\|\big(\sum_{k\geq0}|x_k|^2\big)^{\frac12}\big\|_{\Phi}+\big\|\big(\sum_{k\geq0}|x_k^{\ast}|^2\big)^{\frac12}\big\|_{\Phi}.$$
\end{cor}

\appendix

\section{Generalised martingale transform theorem}

In this appendix, we establish weak $(1,1)$ estimate for the generalised martingale transform. The proof is somewhat similar to that of weak $(1,1)$ estimate for a martingale transform given in \cite[Theorem 3.3.2]{Xu}. The similarity with the proof of Theorem 3.1 in \cite{PR} is less obvious.

\begin{thm}\label{mt 11 theorem} Let $(\mathcal{N},\nu)$ be a noncommutative probability space and let $(\xi_k)_{k\geq0}\subset\mathcal{N}$ be such that $\sup_{k\geq 0}\|\xi_k\|_{\infty}<\infty.$ For every $x\in L_1(\mathcal{M}),$ we have
$$\Big\|\sum_{k\geq0}(\mathcal{E}_{k}x-\mathcal{E}_{{k-1}}x)\otimes\xi_k\Big\|_{L_{1,\infty}(\mathcal{M}\bar{\otimes}\mathcal{N})}\leq 90\|x\|_1\cdot \sup_{k\geq0}\|\xi_k\|_{\infty}.$$
\end{thm}

The following decomposition appeared first as Theorem 2.1 in \cite{PR} (see also \cite[Theorem 3.1]{JRWZ}).

\begin{thm}[Gundy's decompostion]\label{Gundy} Let  $x=x^{\ast}\in L_1(\mathcal{M}).$ For a given $0<\lambda\in\mathbb{R},$ there exist $\alpha,\beta,\gamma,\delta\in L_1(\mathcal{M})$ such that
\begin{enumerate}[{\rm (i)}]
\item\label{gundya} $x=\alpha+\beta+\gamma+\delta;$
\item\label{gundyb} $\alpha\in L_2(\mathcal{M})$ and $\|\alpha\|_2^2 \leq 2\lambda\|x\|_1;$
\item\label{gundyc} $\beta$ satisfies the condition
$$\sum_{k\geq 0}\|\beta_k\|_1 \leq 4 \|x\|_1,$$
where $\beta_k=\mathcal{E}_k\beta-\mathcal{E}_{k-1}\beta$ for $k\geq0;$
\item\label{gundyd} $\gamma$ and $\delta$ satisfy the conditions
$$\tau\Big(\bigvee_{k\geq0} \mathrm{supp}|\gamma_k|\Big)\leq\lambda^{-1}\|x\|_1,\quad \tau\Big(\bigvee_{k\geq0} \mathrm{supp}|\delta_k^{\ast}|\Big)\leq\lambda^{-1}\|x\|_1,$$
where $\gamma_k=\mathcal{E}_k\gamma-\mathcal{E}_{k-1}\gamma,$ $\delta_k=\mathcal{E}_k\delta-\mathcal{E}_{k-1}\delta$ for $k\geq0;$
\end{enumerate}
\end{thm}

\begin{proof}[Proof of Theorem \ref{mt 11 theorem}] Set
$$T_{\xi}x=\sum_{k\geq0}(\mathcal{E}_{k}x-\mathcal{E}_{{k-1}}x)\otimes\xi_k.$$
Without loss of generality, $\xi_k=\xi_k^{\ast}$ for all $k\geq0$ and $\sup_{k\geq0}\|\xi_k\|_{\infty}=1.$

Suppose first that $x=x^{\ast}\in L_1(\mathcal{M})$ is such that $\|x\|_1=1$ and fix $t>0.$ Let $\lambda=t^{-1}$ and let $\alpha,\beta,\gamma,\delta\in L_1(\mathcal{M})$ be the elements given by Theorem \ref{Gundy}. Obviously, we have
$$T_{\xi}x=T_{\xi}\alpha+T_{\xi}\beta+T_{\xi}\gamma+T_{\xi}\delta.$$
Thus,
\begin{equation}\label{mt11 eq0}
\mu(4t,T_{\xi}x)\leq\mu(t,T_{\xi}\alpha)+\mu(t,T_{\xi}\beta)+\mu(t,T_{\xi}\gamma)+\mu(t,T_{\xi}\delta).
\end{equation}

Now,
$$t\mu^2(t,T_{\xi}\alpha)\leq \|T_{\xi}\alpha\|_2^2=(\tau\otimes\nu)\Big(\Big|\sum_{k\geq 0}\alpha_k\otimes \xi_k\Big|^2\Big)=\sum_{k\geq0}\|\alpha_k\|_2^2\|\xi_k\|_2^2.$$
Since $(\mathcal{N},\nu)$ is a noncommutative probability space, it follows that
$$\|\xi_k\|_2^2\leq \|\xi_k\|_{\infty}^2\leq 1,\quad k\geq0.$$
Therefore,
$$t\mu^2(t,T_{\xi}\alpha)\leq \sum_{k\geq0}\|\alpha_k\|_2^2=\|\alpha\|_2^2\stackrel{Th.\ref{Gundy}\eqref{gundyb}}{\leq} 2t^{-1}.$$
We conclude that
\begin{equation}\label{mt11 eq1}
t\mu(t,T_{\xi}\alpha)\leq \sqrt{2}.
\end{equation}

Next,
$$t\mu(t,T_{\xi}\beta)\leq\|T_{\xi}\beta\|_1\leq \sum_{k\geq0}\|\beta_k\otimes\xi_k\|_1=\sum_{k\geq0}\|\beta_k\|_1\|\xi_k\|_1.$$
Since $(\mathcal{N},\nu)$ is a noncommutative probability space, it follows that
$$\|\xi_k\|_1\leq \|\xi_k\|_{\infty}\leq 1,\quad k\geq0.$$
Therefore,
\begin{equation}\label{mt11 eq2}
t\mu(t,T_{\xi}\beta)\leq\sum_{k\geq0}\|\beta_k\|_1\stackrel{Th.\ref{Gundy}\eqref{gundyc}}{\leq}4.
\end{equation}

Next, observe that
$$T_{\xi}\gamma=T_{\xi}\gamma\cdot \Big(\bigvee_{k\geq0} \mathrm{supp}|\gamma_k|\otimes 1\Big).$$
Thus,
$$(\tau\otimes\nu)({\rm supp}(|T_{\xi}\gamma|))\leq (\tau\otimes\nu)\Big(\bigvee_{k\geq0} \mathrm{supp}|\gamma_k|\otimes 1\Big)\stackrel{Th.\ref{Gundy}\eqref{gundyd}}{\leq} t.$$
It follows that
\begin{equation}\label{mt11 eq3}
\mu(t,T_{\xi}\gamma)=0\mbox{ and, similarly, }\mu(t,T_{\xi}\delta)=0.
\end{equation}

Substituting \eqref{mt11 eq1}, \eqref{mt11 eq2} and \eqref{mt11 eq3} into \eqref{mt11 eq0}, we obtain
$$t\mu(4t,T_{\xi}x)\leq 4+\sqrt{2}.$$
Since $t>0$ is arbitrary, it follows that
$$\|T_{\xi}x\|_{1,\infty}\leq 4(4+\sqrt{2})$$
for every $x=x^{\ast}\in L_1(\mathcal{M})$ such that $\|x\|_1=1.$

By homogeneity and the quasi-triangle inequality in $L_{1,\infty}$, we have
$$\|T_{\xi}x\|_{1,\infty}\leq 16(4+\sqrt{2})\|x\|_1,\quad x\in L_1(\mathcal{M}).$$
\end{proof}

The following lemma is a noncommutative version of Marcinkiewicz Interpolation Theorem in the formulation of Calderon (see Theorem IV.4.13 in \cite{BS1988}). The proof is the same as in the commutative setting and is, therefore, omitted.

\begin{lem}\label{weak l1 l2 interpolation lemma} Let $(\mathcal{M},\tau)$ be a noncommutative measure space. If $T:L_1(\mathcal{M})\to L_{1,\infty}(\mathcal{M})$ and simultaneously $T:L_2(\mathcal{M})\to L_2(\mathcal{M}),$ then $T:L_p(\mathcal{M})\to L_p(\mathcal{M})$ and
$$\|T\|_{L_p\to L_p}\leq c_{{\rm abs}}p'\max\{\|T\|_{L_1\to L_{1,\infty}},\|T\|_{L_2\to L_2}\},\quad 1<p\leq 2.$$
\end{lem}

\begin{cor}\label{mt 11 corollary} Let $(\mathcal{N},\nu)$ be a finite von Neumann algebra and let $(\xi_k)_{k\geq0}\subset\mathcal{N}$ be such that $\sup_{k\geq 0}\|\xi_k\|_{\infty}<\infty$. For any $x\in L_p(\mathcal M)$, we have
$$\Big\|\sum_{k\geq0}(\mathcal{E}_{k}x-\mathcal{E}_{{k-1}}x)\otimes\xi_k\Big\|_{L_p(\mathcal M\bar{\otimes} \mathcal N)}\leq c_{{\rm abs}}p'\|x\|_p\cdot \sup_{k\geq0}\|\xi_k\|_{\infty},\quad 1<p\leq 2,$$
where $p'$ is the conjugate index of $p.$
\end{cor}
\begin{proof} Let
$$T_{\xi}x=\sum_{k\geq0}(\mathcal{E}_{k}x-\mathcal{E}_{{k-1}}x)\otimes\xi_k$$
By Theorem \ref{mt 11 theorem}, we have
$$\|T_{\xi}\|_{L_1\to L_{1,\infty}}\leq 90\sup_{k\geq0}\|\xi_k\|_{\infty}.$$
Obviously,
$$\|T_{\xi}\|_{L_2\to L_2}\leq \sup_{k\geq0}\|\xi_k\|_{\infty}.$$
The assertion follows from Lemma \ref{weak l1 l2 interpolation lemma}.
\end{proof}

\noindent {\bf Acknowledgements.}
Authors are grateful to Professor Sergei Astashkin and Dr Konstantin Lykov for finding a mistake in the early version of the paper and for supplying us with a tool \cite{Lykov} to fix this mistake. Authors are grateful to Professor Narcisse Randrianantoanina for his helpful comments. Authors are grateful to Dr Edward McDonald and Mr Thomas Scheckter for their careful reading of the manuscript which led to some certain improvements. Authors are grateful to their home institutions (CSU, UNSW) for never ending support of their research. The last author is grateful to CSU for the hospitality during his research visit when this paper was written.

\end{document}